\documentclass[a4paper, oneside]{amsart}
\usepackage{mathrsfs}
\usepackage{amsfonts}
\usepackage{amssymb}
\usepackage{amsxtra}
\usepackage{dsfont}
\usepackage{enumerate}

\usepackage[english,polish]{babel}

\newcommand{\pl}[1]{\foreignlanguage{polish}{#1}}

\usepackage{color}

\theoremstyle{plain}
\newtheorem{theorem}{Theorem}[section]
\newtheorem{proposition}{Proposition}[section]

\newtheorem{lemma}{Lemma}[section]
\theoremstyle{definition}

\numberwithin{equation}{section}

\theoremstyle{plain}
\newcounter{thm}

\DeclareMathOperator{\Real}{Re}

\newcommand{\RR}{\mathbb{R}}

\newcommand{\ZZ}{\mathbb{Z}}

\newcommand{\CC}{\mathbb{C}}
\newcommand{\NN}{\mathbb{N}}

\newcommand{\calS}{\mathcal{S}}
\newcommand{\calK}{\mathcal{K}}

\newcommand{\calF}{\mathcal{F}}

\newcommand{\calI}{\mathcal{I}}

\newcommand{\calD}{\mathcal{D}}

\newcommand{\ind}[1]{{\mathds{1}_{{#1}}}}
\newcommand{\dist}{\operatorname{dist}}

\renewcommand{\atop}[2]{\substack{{#1}\\{#2}}}

\newcommand{\sprod}[2] {{#1 \cdot #2}}

\newcommand{\la}{\lambda}

\newcommand{\e}{\varepsilon}

\title[On dimension-free variational  inequalities]
{On dimension-free variational  inequalities\\
  for averaging operators in $\mathbb R^d$}

\author{Jean Bourgain}
\address{Jean Bourgain \\
  School of Mathematics\\
  Institute for Advanced Study\\
  Princeton, NJ 08540\\
  USA}
\email{bourgain@math.ias.edu}

\author{Mariusz Mirek}
\address{Mariusz Mirek \\
  School of Mathematics\\
  Institute for Advanced Study\\
  Princeton, NJ 08540\\
  USA \&
  Department of Mathematics\\
  King's College London\\
Strand, London, WC2R 2LS, UK \& 
	Instytut Matematyczny\\
	Uniwersytet \pl{Wroc{\lll}awski}\\
	Plac Grun\-waldzki 2/4\\
	50-384 \pl{Wroc{\lll}aw}\\
	Poland}
\email{mirek@math.ias.edu \& mirek@math.uni.wroc.pl}

\author{Elias M. Stein}
\address{
	Elias M. Stein\\
	Department of Mathematics\\
	Princeton University\\
	Princeton\\
	NJ 08544-100 USA}
\email{stein@math.princeton.edu}

\author{B{\l}a{\.z}ej Wr{\'o}bel}
\address{ B{\l}a{\.z}ej Wr{\'o}bel\\
	Universit\"{a}t Bonn \\
	Mathematical Institute\\
	Endenicher Allee 60\\
	D-53115 Bonn \\
	Germany \&
	Instytut Matematyczny\\
	Uniwersytet \pl{Wroc{\lll}awski}\\
	Pl. Grun\-waldzki 2/4\\
	50-384 \pl{Wroc{\lll}aw}\\
	Poland}
\email{blazej.wrobel@math.uni.wroc.pl}

\thanks{ Jean Bourgain was partially supported by NSF grant DMS-1301619.  Mariusz
  Mirek was partially  supported by the Schmidt Fellowship and the IAS Found for
  Math. and by the National Science Center, NCN grant  DEC-2015/19/B/ST1/01149.
  Elias M. Stein was partially supported by NSF grant
  DMS-1265524.  B{\l}a{\.z}ej Wr{\'o}bel was partially supported by
  the National Science Centre, NCN grant 2014\slash 15\slash D\slash ST1\slash 00405}

\begin{document}
\selectlanguage{english}

\begin{abstract}
  We study dimension-free $L^p$ inequalities for $r$-variations
  of the Hardy--Littlewood averaging operators defined over symmetric
  convex bodies in $\RR^d$.
\end{abstract}

\maketitle

\section{Introduction}\label{sec:0}

The purpose of this paper is to initiate the study of $r$-variational
estimates in the setting of dimension-free bounds for averaging
operators defined over symmetric convex bodies in $\RR^d$.

We will assume that  $G$ is a non-empty convex symmetric body in $\RR^d$, which
simply means that $G$ is  a
bounded convex open and symmetric subset of $\RR^d$. For every
$x\in\RR^d$, $t>0$ and   $f\in L_{{\rm loc}}^1(\RR^d)$ let  
\begin{align}
\label{eq:62}
  M_{t}^Gf(x)=\frac{1}{|G_t|}\int_{G_t}f(x-y){\rm d}y=\frac{1}{|G_t|}\ind{G_t}*f(x)
\end{align}
be the  Hardy--Littlewood averaging operator
defined over the sets
$$G_t=\{y\in\RR^d: t^{-1}y\in G\}.$$

For $r\in[1, \infty)$ the $r$-variation seminorm $V_r$ of
a complex-valued function $(0, \infty)\times\RR^d\ni (t, x)\mapsto\mathfrak a_t(x)$  is defined by
setting 
\[
V_r(\mathfrak a_t(x): t\in Z)=\sup_{\atop{0 < t_0<\ldots <t_J}{t_j\in Z}}
\bigg(\sum_{j=0}^J|\mathfrak a_{t_{j+1}}(x)-\mathfrak a_{t_j}(x)|^r\bigg)^{1/r},
\] 
where $Z$ is a subset of $(0, \infty)$ and the supremum is taken over all finite increasing sequences in
$Z$. Usually $r$-variation $V_r$ defined over the dyadic set $Z=\{2^n:n\in\ZZ\}$ is
called the long $r$-variation seminorm.   
In order to avoid some problems with measurability of
$V_r(\mathfrak a_t(x): t\in Z)$ we assume that $(0, \infty)\ni t\mapsto
\mathfrak a_t(x)$ is always a continuous function for every $x\in\RR^{d}$. The $r$-variational seminorm
is an invaluable tool in pointwise convergence problems. If for some $r\in [1, \infty)$ and $x\in\RR^d$ we have
\[
V_r(\mathfrak a_t(x): t>0)<\infty
\]
then the limits $\lim_{t\to0}a_t(x) $ and $\lim_{t\to\infty}a_t(x)$
exist. So we do not need to establish pointwise convergence on a dense
class, which in many cases is a challenging problem, see \cite{MTS1}
and the references given there.  Moreover, $V_r$'s control the
supremum norm in the following sense, for any $t_0>0$ we have the
pointwise estimate
\[
\sup_{t>0}|\mathfrak a_t(x)|\le |\mathfrak a_{t_0}(x)| + 2V_r(\mathfrak a_t(x): t>0).
\]
There is an extensive literature about estimates for $r$-variational
seminorms. However for our purposes the most relevant will be
\cite{JR1}, \cite{JSW} and \cite{MTS1}, see also the references given there.
\medskip

One of the main results of this paper is the following 
theorem.
\begin{theorem}
  \label{thm:6}
  Let $p\in(3/2, 4)$ and $r\in(2, \infty).$  Then there exists a
  constant $C_{p, r}>0$ independent of the dimension $d\in\NN$ and such that for every symmetric convex body $G\subset\RR^d$ 
 we have
  \begin{align}
   \label{eq:65}
    \big\|V_r\big(M_{t}^Gf: t>0\big)\big\|_{L^p}\le C_{p, r}\|f\|_{L^p}
  \end{align}
  for all  $f\in L^p(\RR^d)$.
\end{theorem}
The range for the parameter $p$ in Theorem \ref{thm:6} can be improved if we consider only long $r$-variations.
Namely, we have the following lacunary variant of Theorem \ref{thm:6}.
\begin{theorem}
  \label{thm:3}
  Let $p\in(1, \infty)$ and $r\in(2, \infty)$. Then there exists a
  constant $C_{p, r}>0$ independent of the dimension $d\in\NN$ and such that for every symmetric convex body $G\subset\RR^d$ 
 we have
  \begin{align}
   \label{eq:64}
    \big\|V_r\big(M_{2^n}^Gf: n\in\ZZ\big)\big\|_{L^p}\le C_{p, r}\|f\|_{L^p}
  \end{align}
for all  $f\in L^p(\RR^d)$.
  \end{theorem}
If we restrict our attention to the balls induced by  small $\ell^q$
norms in $\RR^d$ then we can obtain the full range of $p$'s in Theorem
\ref{thm:6}. To be more precise for
$q\in[1, \infty)$ let us define these balls
\begin{align}
  \label{eq:66}
  \begin{split}
B_q=&\Big\{x=(x_1, \ldots, x_d)\in\RR^d: |x|_q=\Big(\sum_{1\le k\le
  d}|x_k|^q\Big)^{1/q}\le1\Big\}, \qquad \text{ and }\\
B_{\infty}&=\big\{x=(x_1, \ldots, x_d)\in\RR^d: |x|_{\infty}=\max_{1\le k\le d}|x_k|\le1\big\}.
  \end{split}  
\end{align}

Then we have the following theorem.
\begin{theorem}
  \label{thm:7}
  Suppose that $G=B_q$ is one of the balls defined in \eqref{eq:66} for
  some $q\in[1, \infty]$. Let $p\in(1, \infty)$ and $r\in(2, \infty)$. Then there
  exists a constant $C_{p, q, r}>0$ independent of the dimension $d\in\NN$  and such that
  \begin{align}
   \label{eq:67}
    \big\|V_r\big(M_t^Gf: t>0\big)\big\|_{L^p}\le C_{p,q, r}\|f\|_{L^p}
  \end{align}
holds for all  $f\in L^p(\RR^d)$.
  \end{theorem}
The range for parameter $r\in(2, \infty)$ in Theorem \ref{thm:6}, Theorem \ref{thm:3} and Theorem \ref{thm:7} is sharp, see \cite{JSW}.

\medskip

  Dimension dependent versions of Theorems \ref{thm:6} and \ref{thm:3}, with sharp ranges of parameters $p\in(1, \infty)$ and
  $r\in(2, \infty)$, follow from \cite[Theorem A.1]{MTS1}. Also related to our results is the paper of Jones,
  Seeger and Wright \cite{JSW}. Especially, \cite[Theorem 1.4]{JSW},
  where $L^p$ bounds for $r$-variations associated with the spherical
  averages have been established. Their estimates, however, depend on
  the dimension.

For a symmetric convex body $G\subset\RR^d$, $x\in\RR^d$ and $f\in L_{{\rm loc}}^1(\RR^d)$  we set
$$M_{\star}^Gf(x)=\sup_{t>0} M_t^G f(x).$$
The present paper may be thought of as a variational counterpart for a
series of articles establishing dimension-free bounds on $L^p(\RR^d)$
for $M_{\star}^G$ and various symmetric convex bodies $G.$ The
starting point of this line of research was the work of the third
author \cite{SteinMax}, where he obtained the dimension-free bounds on
$L^p(\RR^d),$ with $p\in(1, \infty],$ for $M_{\star}^{B_2},$ where
$B_2$ is the Euclidean ball. The fully corresponding result for
$r$-variations is our Theorem \ref{thm:7} for $G=B_2$, see also
Theorem \ref{thm:ball} in the Appendix. Then, the first author proved
a dimension-free estimate on $L^2(\RR^d)$ for $M_{\star}^G$ when $G$
is a general symmetric convex body, see \cite{B1}. This result has
been independently extended by the first author \cite{B2} and Carbery
\cite{Car1} to $L^p(\RR^d)$ for $p\in(3/2, \infty].$ An
$r$-variational counterpart of \cite{B1}, \cite{B2}, and \cite{Car1}
is our Theorem \ref{thm:6}. Contrary to these papers we also have to
restrict $p$ from above by $4.$ Removing this restriction does not
seem easy, as we do not have a natural endpoint estimate to
interpolate. In \cite{B1}, \cite{B2}, and \cite{Car1}, such an
estimate was the trivial bound for $M_{\star}^G$ on
$L^{\infty}(\RR^d)$. Next, imposing a certain geometric constraint,
M\"uller \cite{Mul1} enlarged the range of $p$'s, for which one has
dimension-free bounds on $L^p(\RR^d)$ for $M_{\star}^G$, to all
$p\in(1, \infty].$ His result includes the cases when $G=B_q$ with
$q\in[1, \infty)$. The most recent development in the study of
dimension-free bounds for averaging operators is \cite{B3}, in which
the first author provided dimension-free estimates on $L^p(\RR^d),$
with $p\in(1, \infty]$ for the cubes $G=B_{\infty}.$
Theorem \ref{thm:7} is a variational counterpart of  the maximal results from \cite{Mul1} and \cite{B3} for
$G=B_q$ with $q\in[1, \infty]$.

Now, let us  describe M\"uller's
result more precisely. Let $G$ be a symmetric convex body in $\RR^d$. By the argument
from \cite{B1}, there exists an invertible linear transformation $U\in{\rm Gl}(\RR^d)$ and a constant
$L(G)>0$ such that
\begin{align}
  \label{eq:84}
  {\rm Vol}_{d}\,U(G)=1 \quad \text{ and } \quad \int_{U(G)}(\xi \cdot x)^2{\rm d}x=L(G)^2
\end{align}
for all unit vectors $\xi\in\mathbb S^{d-1}=\{y\in\RR^d: |y|=1\}$,
where $\xi\cdot x=\langle\xi, x\rangle$ is the standard inner product
in $\RR^d$ and $|y|=|y|_2$ is the Euclidean  norm in $\RR^d$ as in
\eqref{eq:66} with $q=2$.  It is not difficult to note that $L(G)$ is
determined uniquely by \eqref{eq:84}, moreover $U$ is determined
uniquely up to multiplication from the left by an orthogonal
transformation of $\RR^d$.

Similarly as in \cite{B1}, for every $\xi\in\mathbb S^{d-1}$ and $u\in\RR$, we define 
\begin{align}
  \label{eq:86}
  \varphi_{\xi}(u)={\rm Vol}_{d-1}\big(\{x\in U(G): (\xi \cdot x)=u\}\big).
\end{align}
Moreover, we define the constants
\begin{align}
  \label{eq:95}
\sigma(G)^{-1}&=\max\big\{\varphi_{\xi}(0):\xi\in\mathbb S^{d-1}\big\},  \qquad \text{and}\\
\label{eq:96}
  Q(G)&=\max\big\{{\rm Vol}_{d-1}(\pi_{\xi}(U(G))): \xi\in\mathbb S^{d-1}\big\},
\end{align}
where $\pi_{\xi}:\RR^d\to\xi^{\perp}$ denotes the orthogonal projection
of $\RR^d$ onto the hyperplane perpendicular to $\xi$.

We note that $\sigma(V(G))=\sigma(G)$ and $Q(V(G))=Q(G)$ for any
$V\in{\rm Gl}(\RR^d)$. Moreover, in \cite{B1} it was proven that
there is a universal constant $a>0$ such that
$a^{-1}L(G)\le \sigma(G)\le aL(G)$.

Using these two linear invariants $\sigma(G)$ and $Q(G)$ M\"uller 
estimated $\|M_{\star}^G\|_{L^p\to L^p}$. Namely, the main result of
\cite{Mul1} states that for every $p\in(1, \infty]$ and for every
symmetric convex body $G\subset\RR^d$ there is a constant
$C(p,\sigma(G), Q(G))>0$ independent of the dimension $d$ such that
\begin{align}
  \label{eq:100}
  \|M_{\star}^G\|_{L^p\to L^p}\le C(p,\sigma(G), Q(G)).
\end{align}
In other words $\|M_{\star}^G\|_{L^p\to L^p}$ may depend on
$\sigma(G)$ and $Q(G)$, but not explicitly on the dimension $d$. In
fact \cite{B2} and \cite{Car1} show that for $p\in(3/2, \infty]$ the
norm $\|M_{\star}^G\|_{L^p\to L^p}$ can be even chosen independently
of $\sigma(G)$ and $Q(G)$. Using \eqref{eq:100} M\"uller showed that
$\|M_{\star}^G\|_{L^p\to L^p}$ is independent of the dimension for all
$G=B_q$ with $q\in[1, \infty)$, since $\sigma(B_q)$ and $Q(B_q)$ can
be explicitly computed and they are independent of $d$. However, for
the cubes $G=B_{\infty}$ it turned out that $\sigma(B_{\infty})$ is
independent of the dimension, but $Q(B_{\infty})=\sqrt{d}$. The
question about the dimension-free bounds for the cubes in $\RR^d$ for
$p\in(1, 3/2]$ was left open until \cite{B3}, where the first author
significantly refined and extended the methods from \cite{Mul1}. Using the product nature of the cubes he showed that
for every $p\in(1, \infty]$ there is a constant $C_p>0$ independent of the dimension $d$ such that
\begin{equation*}
 % \label{eq:100}
  \|M_{\star}^{B_{\infty}}\|_{L^p\to L^p}\le C_p.
\end{equation*}
The question whether $\|M_{\star}^G\|_{L^p\to L^p}$ with $p\in(1, 3/2]$ can be controlled
by a constant independent of the dimension for  general symmetric convex
bodies remains still open. The situation is even more complicated for $r$-variational estimates.  
A question which we are unable to answer is what happens not only when
$p\in(1, 3/2]$, but also for $p\in[4, \infty)$ in the case of general symmetric convex
bodies. However the method of the proof of Theorem \ref{thm:7}, in view
of M\"uller's result, allows us to deduce that for every
$p\in(1, \infty)$, for every $r\in(2, \infty)$ and for every
symmetric convex body $G\subset\RR^d$  there is a constant 
$C(p,r,\sigma(G), Q(G))>0$ independent of $d$ such that
\begin{align}
  \label{eq:107}
  \sup_{\|f\|_{L^p}\le1}    \big\|V_r\big(M_t^Gf: t>0\big)\big\|_{L^p}\le C(p,r,\sigma(G), Q(G)).
\end{align}
Theorem \ref{thm:6} and Theorem \ref{thm:7} give some partial evidence
to support the conjecture which asserts that for all $p\in(1, \infty)$ and
$r\in(2, \infty)$ and for every general symmetric convex body
$G\subset\RR^d$ the left-hand side of \eqref{eq:107} can be controlled
by a constant independent of the dimension and the linear invariants
$\sigma(G)$ and $Q(G)$.

\medskip

Let us briefly describe the strategy for proving our main results. The
first step is splitting the consideration into long and short
variations, see \eqref{eq:21}. The long variations are treated in Theorem \ref{thm:3} by
appealing to known dimension-free estimates for $r$-variations of the
Poisson semigroup (see \cite[Theorem 3.3]{JR1}) together with a square
function estimate. The square function estimate is proved by using a
dimension-free Littlewood--Paley theory, which allows us to prove very
good estimates on $L^2$ and acceptable estimates on $L^p.$ Then
interpolation establishes Theorem \ref{thm:3}.  In view of Theorem
\ref{thm:3} in order to prove Theorems \ref{thm:6} and \ref{thm:7} it
is enough to consider short variations. To this end we also use a
dimension-free Littlewood--Paley theory and interpolation between $L^2$
and $L^p$ bounds. The analysis of short variations breaks basically
into two cases, whether $p\in[2, \infty]$ or $p\in(1, 2]$. In the
first case we rely on the numerical inequality, Lemma \ref{lem:6},
which reduces estimates for $r$-variations to the situation where the
division intervals over which differences are taken are all of the
same size. The case $r=2$ of this is particularly suited to an
application of the Fourier transform. On the other hand, when
$p\in(1, 2]$ we use an orthogonality principle (an $r$-variation
adaptation of an idea in \cite{Car1}), together with an appropriate
characterization of $r$-variation in terms of fractional derivatives,
given in Proposition \ref{prop:chVr}. Of course fractional derivatives
had already occurred in \cite{Car1} and \cite{Mul1}, as well as the
earlier proof of the spherical maximal theorem for $d\ge3$. It should
be pointed out that for $p\in(1, 2)$ it is essential that we can use
$r$-variational estimates for $r\in(1, 2)$.

\subsection{Applications} The results from the previous paragraph have an ergodic theoretical interpretation. Namely,
let $(X, \mathcal B, \mu)$ be a $\sigma$-finite measure space with
families of commuting and measure-preserving transformations
$(T_1^t: t\in\RR)$, $(T_2^t: t\in\RR), \ldots, (T_d^t: t\in\RR)$, which map
$X$ to itself. For every symmetric convex body $G\subset\RR^d$ for every
$x\in X$ and $f\in L^1(X, \mu)$ we define the ergodic
Hardy--Littlewood averaging operator by setting
\begin{align}
  \label{eq:22}
  A_t^Gf(x)=\frac{1}{|G_t|}\int_{G_t}f\big(T_1^{y_1}\circ T_2^{y_2}\circ\ldots\circ T_d^{y_d}x\big){\rm d}y_1{\rm d}y_2\ldots{\rm d}y_d.
\end{align}
For this operator we also have dimension free $r$-variational estimates.
\begin{theorem}
\label{thm:5}
 Let $p\in(3/2, 4)$ and $r\in(2, \infty).$  Then there exists a
  constant $C_{p, r}'>0$ independent of the dimension $d\in\NN$ and such that for every symmetric convex body $G\subset\RR^d$ 
 we have
  \begin{align}
\label{eq:43}
    \big\|V_r\big(A_{t}^Gf: t>0\big)\big\|_{L^p(X, \mu)}\le C_{p, r}'\|f\|_{L^p(X, \mu)}
  \end{align}
  for all  $f\in L^p(X, \mu)$. Moreover, if we  consider only long variations, then \eqref{eq:43} remains true for all 
  $p\in(1, \infty)$ and $r\in(2, \infty)$ and we have
 \begin{align}
\label{eq:60}
    \big\|V_r\big(A_{2^n}^Gf: n\in\ZZ\big)\big\|_{L^p(X, \mu)}\le C_{p, r}'\|f\|_{L^p(X, \mu)}.
  \end{align} 
\end{theorem}
Theorem \eqref{thm:5} is an ergodic counterpart of Theorem \ref{thm:6}
and Theorem \ref{thm:3}. If $G=B_q$ for some $q\in[1, \infty]$ then we
obtain sharp ranges of exponents with respect to the parameters $p$ and
$r$ for the operator \eqref{eq:22}. Namely, we can prove an analogue of
Theorem \ref{thm:7}.
\begin{theorem}
\label{thm:2}
  Suppose that $G=B_q$ is one of the balls defined in \eqref{eq:66} for
  some $q\in[1, \infty]$. Let $p\in(1, \infty)$ and $r\in(2, \infty)$. Then there
  exists a constant $C_{p, q, r}'>0$ independent of the dimension $d\in\NN$  and such that
  \begin{align}
\label{eq:63}
    \big\|V_r\big(A_t^Gf: t>0\big)\big\|_{L^p(X, \mu)}\le C_{p,q, r}'\|f\|_{L^p(X, \mu)}
  \end{align}
holds for all  $f\in L^p(X, \mu)$.
  \end{theorem}

  In Section \ref{sec:tp} we prove a transference principle (see Proposition \ref{prop:3}), which allows us to deduce 
inequalities  \eqref{eq:43}, \eqref{eq:60} and \eqref{eq:63} from the corresponding estimates  
in  \eqref{eq:65}, \eqref{eq:64}, and \eqref{eq:67}, respectively.

The remarkable feature of Theorem \ref{thm:5} and Theorem
  \ref{thm:2} is that the implied bounds in \eqref{eq:43},
  \eqref{eq:60}, and \eqref{eq:63} are independent of the number of
  underlying commuting and measure-preserving transformations $T_1^{y_1},\ldots,T_d^{y_d}$.
  On the other hand, due to the properties of $r$-variational seminorm we immediatley know that
  the limits
  \[
\lim_{t\to0}A_t^Gf(x),\qquad \text{and} \qquad \lim_{t\to\infty}A_t^Gf(x)
  \]
  exist almost everywhere on $X$ for every $f\in L^p(X, \mu)$ and
  $p\in(1, \infty)$. It is worth  emphasizing that although the
  pointwise convergence for $M_t^Gf$ as $t\to0$ (respectively
  $t\to\infty$) can be easily deduced from the maximal bounds, this is
  much harder for $A_t^Gf$. For $M_t^Gf$ there are many natural
  dense subspaces which could be used to establish pointwise
  convergence. However, for $A_t^Gf$, which is\ defined on an abstract measure
  space there is no obvious way how to even find such a candidate for a
  dense class. Fortunately, $r$-variational estimates allow us to
  obtain the desired conclusion directly. 

  At this stage a  similar question  concerning the discrete analogue of \eqref{eq:22} arises. One  can ask about $r$-variational estimates independent of the dimension for the following  operator 
  \begin{align*}
      \mathcal A_t^Gf(x)=\frac{1}{|G_t\cap \ZZ^d|}\sum_{y\in G_t\cap\ZZ^d}f\big(T_1^{y_1}\circ T_2^{y_2}\circ\ldots\circ T_d^{y_d}x\big).
  \end{align*}
  This immediately lead us to the averaging operators  on $\ZZ^d$, i.e.
  \begin{align*}
      \mathcal M_t^Gf(x)=\frac{1}{|G_t\cap \ZZ^d|}\sum_{y\in G_t\cap\ZZ^d}f(x-y).
  \end{align*}
  However, nothing is known in the discrete setup, even the dimension free
  estimates for the maximal functions
  $\mathcal M_{\star}^Gf(x)=\sup_{t>0}|\mathcal M_t^Gf(x)|$
  are unknown. There is no easy way to derive these
  estimates from the corresponding dimension free estimates of the continuous
  analogue of $\mathcal M_{\star}^Gf$. In the ongoing project \cite{BMSW2} we initiated investigations in this direction and 
  we are able to provide dimension free bounds for the discrete maximal functions $\mathcal M_{\star}^Gf$ on $\ell^p(\ZZ^d)$ in a certain
  range of parameters $p$. We only handled  the case of the Euclidean
  balls $G=B_2$ and the cubes $G=B_{\infty}$. In the case of the discrete
  cubes we can also provide $r$-variational estimates (in some range of $p$ and $r$) and the
  results from this paper find applications there.
\subsection{Notation and basic reductions}
Let $\mathcal S(\RR^d)$ denote the set of all Schwartz functions on
$\RR^d$. To prove our results, by a simple density argument, it
suffices to establish inequalities \eqref{eq:65}, \eqref{eq:64} and
\eqref{eq:67} for all $f\in\mathcal S(\RR^d)$. From now on, in the proofs
presented in the paper $f$ will be always a Schwartz function.

Let $U$ be an invertible linear transformation of $\RR^d$ and let
$U_{p}$ be the isometry of $L^p$ given by
\[
  U_{p}f(x)=|\det U|^{-1/p}f(U^{-1}x) \quad \text{for any}\quad  p\ge1.
\]
Then we have
$$	U_{p}\big(V_r\big(  M_t^G f: t>0\big)\big)
=V_r\big( U_{p}(M_t^G f): t>0\big)= V_r\big( M_t^{U(G)}(U_{p}f):
t>0\big),$$
since
\[
  U_{p} \circ M_t^{G}=M_t^{U(G)} \circ U_{p}.
\]
Therefore $G$ in \eqref{eq:62} can be freely replaced with any other
symmetric convex body $U(G)$ and $L^p$ bounds remain unchanged. This
is an important remark which allows us to assume that \eqref{eq:86}
always holds.  From now on we will assume that $U(G)=G$. More
precisely we will assume that ${\rm Vol_d}\,G=|G|=1$ and that $G$ is
in the isotropic position, which means that there is an isotropic
constant $L=L(G)>0$ such that for every unit vector
$\xi\in\mathbb S^{d-1}$ we have
         \begin{align}
           \label{eq:68}
          \int_G (\xi \cdot x)^2{\rm d}x= L(G)^2.
         \end{align}

As in \cite{B1} the Fourier methods will be extensively exploited here 
to establish $L^2$ bounds in the aforementioned theorems. Let us define the Fourier transform  $\calF$ of a function $f \in
\mathcal S(\RR^d)$ by setting 
$$
\calF f(\xi)=\calF_{\RR^d} f(\xi) = \int_{\RR^d} f(x) e^{-2\pi i \sprod{\xi}{x}} {\: \rm d}x
$$
for any $\xi\in\RR^d$ and let $\mathcal F^{-1}=\mathcal F^{-1}_{\RR^d}$  denote the inverse
Fourier transform on $\RR^d$.

Since $|G|=1$ we see that the kernel of \eqref{eq:62} satisfies
\[
  |G_t|^{-1}\ind{G_t}(x)=t^{-d}K_G(t^{-1} x)
\]
for all $t>0$, where $K_G(x)=\ind{G}(x)$ and 
\[
  \calF (|G_t|^{-1}\ind{G_t})(\xi)=m^G(t\xi),
\]
where
\[
  m^G(\xi)=\calF(K_G)(\xi)=\calF(\ind{G})(\xi).
\]
The isotropic position of $G$ allows us to provide dimension-free
estimates for the multiplier $m$.

\begin{proposition}[{\cite[eq. (10),(11),(12)]{B1}}]
		\label{prop:1}
		Given  a symmetric convex body  $G\subset \RR^d$ with volume one,
                which is in the isotropic
                position, there exists a
                constant $C>0$ such that for every $\xi\in\RR^d$ we
                have
                \begin{align}
                  \label{eq:69}
                  |m^G(\xi)|\leq C(L |\xi|)^{-1},\qquad  |m^G(\xi)-1|\leq CL
                |\xi|,\qquad   |\langle\xi,\nabla m^G(\xi)\rangle|\le C.
                \end{align}
		The constant $L=L(G)$ is defined  in \eqref{eq:68}, while $C$ is a
                universal constant which does not depend on $d$.

	\end{proposition}

Throughout the whole paper $d\in\NN$ denotes the dimension. Unless otherwise stated $C>0$
will stand for an absolute constant whose value may vary from occurrence
to occurrence and  it will never depend on the dimension.  We will
use the convention that $A \lesssim_{\delta} B$
($A \gtrsim_{\delta} B$) to say that there is an absolute constant
$C_{\delta}>0$ (which possibly depends on $\delta>0$) such that
$A\le C_{\delta}B$ ($A\ge C_{\delta}B$). We abbreviate $A\lesssim B$
when the implicit constant is independent of  $\delta$. We will write
$A \simeq_{\delta} B$ when $A \lesssim_{\delta} B$ and
$A\gtrsim_{\delta} B$ hold simultaneously. 

For simplicity of the notation we will often abbreviate $M_t=M_t^G$ and $m=m^G.$

\section*{Acknowledgements}
The authors are grateful to the referees for careful reading of the manuscript and useful remarks
that led to the improvement of the presentation.

\section{Useful tools}
\label{sec:2}

In this section we gather some general tools which will be used in the
proofs of our main results. 

\subsection{Properties of $r$-variations}
We begin with some simple properties of $r$-variation seminorms.
For $r\in[1, \infty)$ the $r$-variation seminorm $V_r$ of
a complex-valued function $(0, \infty)\ni t\mapsto\mathfrak a_t$  is defined by
\[
V_r(\mathfrak a_t: t\in Z)=\sup_{J\in\NN}V_r^J(\mathfrak a_t: t\in Z),
\]
where $Z\subseteq(0, \infty)$ and
\[
V_r^J(\mathfrak a_t: t\in Z)=\sup_{\atop{0 < t_0<\ldots <t_J}{t_j\in Z}}
\bigg(\sum_{j=0}^J|\mathfrak a_{t_{j+1}}-\mathfrak a_{t_j}|^r\bigg)^{1/r}
\]
and the supremum is taken over all finite increasing sequences in $Z$ of length $J+1$.

The function $r \mapsto V_r(\mathfrak a_t: t\in Z)$ is non-increasing
and for any $Z_1\subseteq Z_2$
satisfies
\[
V_r(\mathfrak a_t: t\in Z_1)\le V_r(\mathfrak a_t: t\in Z_2).
  \]
  Moreover, for every $t_0\in Z$ we have
  \begin{align}
    \label{eq:70}
    \sup_{t\in Z}|\mathfrak a_t|\le|\mathfrak a_{t_0}|+ 2V_r(\mathfrak a_t: t\in Z).
  \end{align}
  If $Z$ is a countable set, then
  \begin{align}
    \label{eq:10}
    V_r(\mathfrak a_t: t\in Z)\le 2\Big(\sum_{t \in Z} |\mathfrak a_t|^r\Big)^{1/r}.
  \end{align}
Finally, for every $r\in[1, \infty)$ there exists $C_r>0$ such that	
\begin{equation}
\label{eq:21}
V_r(\mathfrak a_t: t>0)\leq C_rV_r(\mathfrak a_{2^n}: n\in \ZZ)+C_r\Big(\sum_{n\in\ZZ}
V_r\big(( \mathfrak a_t-\mathfrak a_{2^n}): t\in[2^n, 2^{n+1})\big)^r\Big)^{1/r}.
\end{equation}
The first  quantity on the right side in \eqref{eq:21} is called the
long variation seminorm, whereas the second is called  the short
variation seminorm. 
This is a very useful inequality which, in view of Theorem
\ref{thm:3}, will  permit us to reduce the proofs
of Theorem \ref{thm:6} and Theorem \ref{thm:7}  to the estimates of short
variations associated with $M_t^G$. In order to deal with the short
variations efficiently we will prove
an elementary inequality \eqref{eq:20} which allows us to dominate each dyadic block in the short
variations by suitable
square functions, which are simpler objects to handle.   
\begin{lemma}
	\label{lem:6}
	Given $r \in [1, \infty)$, $n\in\ZZ,$ and a continuous function $\mathfrak
        a:[2^n,
        2^{n+1}]\to\CC$, we have
        \begin{align}
          \label{eq:20}
          \begin{split}
        	V_r\big(\mathfrak a_t: t\in[2^n, 2^{n+1})\big)
		&\leq
		2^{1-1/r}
		\sum_{m\ge-n}
		\Big(
		\sum_{k = 0}^{2^{m+n}-1}
		\big|\mathfrak a_{2^n+{(k+1)}/{2^m}} - \mathfrak a_{2^n+{k}/{2^m}}\big|^r
		\Big)^{1/r}\\
                &=2^{1-1/r}
		\sum_{l\ge0}
		\Big(
		\sum_{k = 0}^{2^{l}-1}
		\big|\mathfrak a_{2^n+{2^{n-l}(k+1)}} - \mathfrak a_{2^n+{2^{n-l}k}}\big|^r
		\Big)^{1/r}.
          \end{split}          
        \end{align}
\end{lemma}
\begin{proof}
 First of all we observe  that any interval
  $[s, t)$ for some real numbers $2^n\le s<t< 2^{n+1}$ can be written as a disjoint
  sum (possibly infinite sum) of dyadic intervals, i.e.
  \begin{align}
    \label{eq:18}
[s, t)=\bigcup_{l\in\ZZ}[w_l, w_{l+1})    
  \end{align}
  with the following properties:
  \begin{itemize}
  \item Each $[w_l, w_{l+1})\in \mathcal I_m$ for some $m\ge-n$, where
\begin{align*}
\mathcal I_m=\Big\{\Big[2^n+\frac{k}{2^m}, 2^n+\frac{k+1}{2^m}\Big): 0 \leq k \leq
2^{m+n}-1\Big\} .
\end{align*}
\item There are at most two dyadic intervals on the right-hand side of \eqref{eq:18} with the
  same length.
  \end{itemize}

To prove \eqref{eq:18} let us
consider  dyadic intervals of maximal length contained in
$A=[s, t)$. There are  at most two such intervals. Let $I_0$ be the
one which lies closer to the left endpoint $s$. Then $[s, t)\setminus I_0$ is a sum of
at most two intervals $A_1, B_1$. Without loss of generality we may
assume that $s\in A_1$. Now let $I_1$ be a dyadic interval
contained in $A_1$ with maximal length such that $\dist(I_1,
I_0)=0$. Then by the maximality of $I_0$ we see that
$|I_1|\le |I_0|/2$. Now we define $A_2=A_1\setminus I_1$. If
$A_2\not=\emptyset$ then we proceed as in the previous step. Namely we
take a dyadic interval $I_2$ contained in $A_2$ with maximal length
such that $\dist(I_2, I_1)=0$. By the maximality of $I_1$ we see that
$|I_2|\le |I_1|/2$. Now we define $A_3=A_2\setminus I_2$. If
$A_3\not=\emptyset$ then we proceed likewise above. Then inductively
we obtain a sequence of disjoint dyadic intervals $(I_j: j\in\NN)$ such that
$|I_1|>|I_2|>\ldots$ and $A_1=\bigcup_{j\in\NN}I_j$. If $B_1$ is empty
we are done. If not then we can repeat the argument as for $A_1$ and
obtain a sequence of disjoint dyadic intervals $(J_j: j\in\NN)$ contained in
$B_1$ such that $|J_1|>|J_2|>\ldots$ and
$B_1=\bigcup_{j\in\NN}J_j$. Thus we have proven
that
\[
[s, t)=\bigcup_{j\in\NN}I_j\cup I_0\cup \bigcup_{j\in\NN}J_j.
  \]
From the construction described above it is clear that
there are at most two dyadic intervals in the sum which have the same
length.

Having proven \eqref{eq:18} we can show \eqref{eq:20}. Namely, let
$2^n\le t_0 < t_1 < \ldots < t_J < 2^{n+1}$ be any increasing
sequence. By \eqref{eq:18} for any integer $0\le j<J$ we write
	$$
	[t_j, t_{j+1}) = \bigcup_{l\in\ZZ} [w_l^j, w_{l+1}^j),
	$$
	where each  $[w_l^j, w^j_{l+1})$ is a dyadic interval which
        belongs to $\mathcal I_m$ for some $m\ge -n$. Thus
        \begin{align*}
          |\mathfrak a_{t_{j+1}} -\mathfrak a_{t_j}|\le\sum_{l\in\ZZ}
          |\mathfrak a_{w_{l+1}^j} -\mathfrak a_{w_l^j}
          |
          =
          \sum_{m\ge-n}
          \sum_{l\in\ZZ:\: [w_l^j, w_{l+1}^j) \in \calI_m}
          |\mathfrak a_{w_{l+1}^j} -\mathfrak a_{w_l^j}|	  
        \end{align*}
        and $|\{l\in\ZZ:\ [w_l^j, w_{l+1}^j) \in \calI_m\}|\le 2$ for
        any $m\ge-n$. 
        Hence, we obtain
	\begin{align*}
		\Big(
		\sum_{j = 0}^{J-1}
		|\mathfrak a_{t_{j+1}} -\mathfrak a_{t_j}|^r
		\Big)^{1/r}
		&\leq
		\Big(
		\sum_{j = 0}^{J-1}
		\Big(
		\sum_{m\ge-n}
		\sum_{l\in\ZZ:\: [w_l^j, w_{l+1}^j) \in \calI_m}
		|\mathfrak a_{w^j_l} -\mathfrak a_{w^j_{l+1}}|
		\Big)^r
		\Big)^{1/r}\\
		&\leq
		\sum_{m\ge-n}
		\Big(
		\sum_{j = 0}^{J-1}
		\Big(
		\sum_{l\in\ZZ:\: [w_l^j, w_{l+1}^j) \in \calI_m}
		|\mathfrak a_{w^j_l} -\mathfrak a_{w^j_{l+1}}|
		\Big)^r
		\Big)^{1/r} && \text{by triangle inequality}\\
                &\le 2^{1-1/r}
		\sum_{m\ge-n}
		\Big(
		\sum_{j = 0}^{J-1}
		\sum_{l\in\ZZ:\: [w_l^j, w_{l+1}^j) \in \calI_m}
		|\mathfrak a_{w^j_l} -\mathfrak a_{w^j_{l+1}}|^r
		\Big)^{1/r}&& \text{by H\"older's inequality}.
              \end{align*}
              The last sum is dominated by the right-hand side of
              \eqref{eq:20}, since $[t_j, t_{j+1})\cap[t_{j'},
              t_{j'+1})=\emptyset$ for $j\not=j'$ and the proof is
              completed. 
\end{proof}

We finish this section with an approximate characterization of the
$r$-variation seminorm, which is interesting in its own right. In the
proposition stated below we will be dealing with functions $F$ defined
on $\RR$ that are compactly supported and belong to $L^r(\RR)$ for
some $r\in[1, \infty]$. For such a function $F$ we will write a fractional
derivative $D^{\alpha}F$  defined   as a tempered distribution
by the formula
\begin{align*}
  \calF_{\RR}(D^{\alpha}F)(\xi)&=(2\pi |\xi|)^{\alpha}\calF_{\RR}F(\xi),
\end{align*}
for every $\xi\in\RR$ and $\alpha\in[0, 1]$.
Here  $\calF_{\RR}$ stands for the Fourier transform on $\RR$ given by
  $$\calF_{\RR}f(\xi)=\int_{\RR} f(x) e^{-2\pi i {\xi}{x}} {\rm d}x.$$
  \begin{proposition}
	\label{prop:chVr}
	Let $F$ be a complex-valued function with a compact support in $\RR$.
	\begin{enumerate}[(i)]
        \item Suppose that $F$ and $D^{\alpha} F$ are in $L^r(\RR)$
        for some $r\in[1, \infty]$ and $\alpha>1/r.$ Then there is
        $C_r>0$ such that
		\begin{equation}
		\label{eq:chVr1}
                V_r(F(t): t\in\RR)\le C_r\big(\|F\|_{L^r(\RR)}+\|D^{\alpha }F\|_{L^r(\RR)}\big).
		\end{equation}  
			\item  Conversely, assume that $F\in L^r(\RR)$ and $V_r(F(t): t\in\RR)<\infty.$ Then for every $\beta<1/r$ there is $C_{\beta, r}>0$ such that
			\begin{equation}
			\label{eq:chVr2}
		\|D^{\beta }F\|_{L^r(\RR)}\le C_{\beta,r}\big(\|F\|_{L^r(\RR)}+V_r(F(t): t\in\RR)\big).
			\end{equation} 
			 
		\end{enumerate}
\end{proposition}
\begin{proof}
We begin with demonstrating part (i).  Let $J_{\alpha}$ be the Bessel potential operator
defined for any $h\in L^r(\RR)$ by
\[
J_{\alpha}(h)(x)=G_{\alpha}*h(x),
\]
where
\[
\mathcal F_{\RR}G_{\alpha}(\xi)=\big(1+4\pi|\xi|^2\big)^{-\alpha/2}.
\]
It is known from \cite[Chapter 5, Section 3.2]{StSing} that there are
finite measures $\nu_{\alpha}$ and $\la_{\alpha}$ such that
\begin{align*}
  \big(1+4\pi^2 |\xi|^2\big)^{\alpha/2}=\hat{\nu}_{\alpha}(\xi)+(2\pi |\xi|)^{\alpha}\hat{\la}_{\alpha}(\xi).
\end{align*}
Thus one can represent $F$ in terms of $J_{\alpha}$. Namely,
$F=J_{\alpha}(f),$ where
$$f=F*_{\RR}\nu_{\alpha}+(D^{\alpha} F)*_{\RR}\la_{\alpha},$$
and by our assumptions $f\in L^r(\RR)$. Moreover, for every $t\in \RR$
we have
$$F(t)=\int_{\RR} G_{\alpha}(u)f(t-u){\: \rm d}u.$$

For the Bessel kernel $G_{\alpha}$ one has the following estimates
(see \cite[Chapter 5, Section 3.1]{StSing}, here $n=1$)
\begin{align}
  \label{eq:35}
  G_\alpha(u)=O\big(|u|^{-1+\alpha}\big)\qquad \textrm{and}\qquad \textrm{$G_{\alpha}(u)$ is rapidly decreasing at infinity},
\end{align}
and
\begin{align}
  \label{eq:36}
  \frac{{\rm d}}{{\rm d}u}G_\alpha(u)=O\big(|u|^{-2+\alpha}\big)\qquad \textrm{and}\qquad \textrm{$\frac{{\rm d}}{{\rm d}u}G_\alpha(u)$ is rapidly decreasing at infinity}.
\end{align}
Note that this representation of $F$, and the properties of
$G_{\alpha}$ stated immediately above, show that $F$ may be taken to
be continuous on $\RR$ and not merely defined almost-everywhere, and
hence $V_r(F(t): t\in\RR)$ is well-defined.
Now for any $s=\sigma+i\tau$ consider $F_s$ defined by
$$F_s(t)=e^{(s-1+1/r)^{2}}\int_{\RR} G_{\alpha}(u)|u|^{1-s-1/r}f(t-u){\rm d}u.$$
 When $s=1-1/r,$ then $F_s(t)=F(t).$ Next, we show that
\begin{equation}
\label{eq:chVraux1}
V_1(F_s(t): t\in\RR)\le C_0 \|f\|_{L^1(\RR)},\qquad \text{if} \quad \Real(s)=0.
\end{equation}
Indeed, when $\Real (s)=0,$ then the derivative of $F_s$ (more precisely, here we mean the weak derivative) can be estimated by
$$\bigg|\frac{{\rm d} }{{\rm d} t}F_s(t)\bigg|\le ce^{-\tau^2}\int_{\RR}\bigg|\frac{{\rm d}}{{\rm d}u}\big(G_\alpha(u)|u|^{1-s-1/r}\big)\bigg||f(t-u)|{\rm d}u.$$
Thus taking into account the properties of $G_{\alpha}$ stated in \eqref{eq:35} and \eqref{eq:36} and the fact that $\alpha>1/r$ (which imply that $F_s\in L^{\infty}(\RR)$ and both  $F_s$ and $\frac{{\rm d} }{{\rm d} t}F_s$ are in $L^1(\RR)$), we obtain 
$$V_1(F_s(t): t\in\RR)\le\int_{\RR}\bigg|\frac{{\rm d} }{{\rm d} t}F_s(t)\bigg|{\rm d}t \le C_0 \|f\|_{L^1(\RR)},$$
where
$$C_0:=c'(1+|\tau|)e^{-\tau^2}\bigg(\int_{|u|\le 1}|u|^{-1+\alpha-1/r}{\rm d}u+\int_{|u|>1}|u|^{-2}{\rm d}u\bigg)<\infty.$$
On the other hand, we show
\begin{equation}
\label{eq:chVraux2}
V_{\infty}(F_s(t): t\in\RR)\le 2C_1 \|f\|_{L^{\infty}(\RR)},\qquad \text{if} \quad \Real(s)=1.
\end{equation}
Indeed,
when $\Real(s)=1$, then
$$\sup_{t\in \RR} |F_s(t)|\le e^{1/r^2}\int_{\RR} |G_{\alpha}(u)| |u|^{-1/r}{\rm d}u\, \|f\|_{L^{\infty}(\RR)},$$
since
$$C_1:=e^{1/r^2}\int_{\RR}|G_{\alpha}(u)||u|^{-1/r}{\rm d}u \le c \bigg(\int_{|u|\le 1}|u|^{-1+\alpha-1/r}{\rm d}u+\int_{|u|>1}|u|^{-2}{\rm d}u\bigg)<\infty.$$

Now the mappings $f\mapsto F_s$ can be rephrased as an analytic family
of operators as follows. Choose any sequence $t_0<t_1<\cdots <t_N$ and
define $T_s(f)$ to be the sequence
\[
\big(F_s(t_k)-F_s(t_{k-1}): k\in\ZZ_N\big),
\]
where $\ZZ_N=\{1, 2,\ldots, N\}$.
Observe now that \eqref{eq:chVraux1} and
\eqref{eq:chVraux2} imply that
$$\|T_s(f)\|_{\ell^1(\ZZ_N)}\le C_0 \|f\|_{L^1(\RR)},\qquad \text{if}\quad \Real(s)=0 $$
and
$$ \|T_s(f)\|_{\ell^{\infty}(\ZZ_N)}\le 2C_1 \|f\|_{L^{\infty}(\RR)},\qquad \text{if}\quad \Real(s)=1.$$
Now the complex interpolation theorem, see \cite[Chapter 5, Section 4]{SteWei}, shows that for $s=1-1/r,$
$$\|T_s(f)\|_{\ell^r(\ZZ_N)}\le C\|f\|_{L^r(\RR)}.$$
As a result, since $F_s=F$ for $s=1-1/r$, we obtain
$$\Big(\sum_{k=1}^N |F(t_k)-F(t_{k-1})|^r\Big)^{1/r}\le C \|f\|_{L^r(\RR)}.$$
Since the constant $C$ depends only on the constants $C_0,$ $C_1,$ and
$r,$ therefore $C$ is independent of the choice of $(t_k: 0\le k\le N)$
and $N$. This gives
$$V_r(F(t): t\in\RR)\le C\|f\|_{L^r(\RR)}$$
and combined with
$$\|f\|_{L^r(\RR)}
\le c(\|F\|_{L^r(\RR)}+\|D^{\alpha} F\|_{L^r(\RR)})$$
proves part (i) and \eqref{eq:chVr1}.

We now focus on part (ii). We can assume that $r<\infty$ for otherwise
there is nothing to prove. Let $h\ge 0$ and  $A=V_r(F(t):t\in\RR)<\infty,$ then
\begin{equation}
\label{eq:chVraux3} \sum_{k\in \ZZ} |F(v+(k+1)h)-F(v+kh)|^r\le A^r,
\end{equation}
Now integrate this inequality for $v$ in the interval $[0,h].$
Since 
$$\int_0^h \big|F(v+(k+1)h)-F(v+kh)\big|^r{\rm d}v=\int_{kh}^{(k+1)h}|F(v+h)-F(v)|^r{\rm d}v,$$
inserting this in \eqref{eq:chVraux3} gives the modulus of continuity inequality
$$
\|F(\,\cdot+h)-F(\cdot)\|_{L^r(\RR)}\le A h^{1/r},\qquad \text{for} \quad h\ge 0,$$
from which the result for negative $h$ also follows yielding
\begin{equation}
\label{eq:chVraux4}\|F(\, \cdot+h)-F(\cdot)\|_{L^r(\RR)}\le V_r(F(t):t\in\RR)\,|h|^{1/r},\qquad \text{for} \quad h\in \RR.
\end{equation}
We now invoke the fact that for any $0<\beta<1$
$$D^{\beta}F(t)=c_{\beta}\lim_{\e\to 0}\int_{|h|\ge \e}\, \frac{F(t+h)-F(t)}{|h|^{1+\beta}}{\rm d}h,$$
for a suitable constant, where the limit is taken in the sense of distributions (see \cite[Chapter 5, 6.10]{StSing} and the references given there). Writing
$$D^{\beta}F(t)=c_{\beta}\lim_{\e\to 0}\int_{1\ge |h|\ge \e}\, \frac{F(t+h)-F(t)}{|h|^{1+\beta}}{\rm d}h+c_{\beta}\int_{|h|>1}\, \frac{F(t+h)-F(t)}{|h|^{1+\beta}}{\rm d}h,$$
we see that both terms belong to $L^r(\RR)$.  
To estimate the first term we apply Minkowski's integral inequality and \eqref{eq:chVraux4} and obtain
\begin{align*}
  \bigg(\int_{\RR}\bigg|\int_{1\ge |h|\ge \e}\, \frac{F(t+h)-F(t)}{|h|^{1+\beta}}{\rm d}h\bigg|^r{\rm d}t\bigg)^{1/r}&\le  V_r(F(t): t\in\RR)\int_{1\ge |h|\ge \e}|h|^{-1-\beta+1/r}{\rm d}h\\
  &\le C_{\beta} V_r(F(t): t\in\RR),
\end{align*}
since $\beta <1/r.$
The second term is controlled by
\begin{equation*}
\bigg(\int_{\RR}\bigg|\int_{|h|\ge 1}\, \frac{F(t+h)-F(t)}{|h|^{1+\beta}}{\rm d}h\bigg|^r{\rm d}t\bigg)^{1/r}\le  2\|F\|_{L^r(\RR)}\int_{|h|\ge 1}|h|^{-1-\beta}{\rm d}h \le C'_{\beta} \|F\|_{L^r(\RR)},
\end{equation*}
because $\beta>0$ and $\|F(t+h)\|_{L^r(\RR)}=\|F\|_{L^r(\RR)}.$ This proves part (ii) and \eqref{eq:chVr2}.
\end{proof}

\subsection{An almost orthogonality principle}
Now we prove an $r$-variational counterpart of an almost orthogonality principle 
 from \cite{Car1}. Proposition \ref{prop:2} will be a
key ingredient in the proofs of Theorem \ref{thm:6} and Theorem \ref{thm:7} for $p<2$. 
\begin{proposition}
  \label{prop:2}
  Suppose that $(X, \mathcal B,\mu)$ is a $\sigma$-finite measure
  space and let $(T_{t}: t\in Z)$ be a family of positive linear
  operators\footnote{We say that a linear operator $T$ is positive, if
    $Tf\ge0$ for every function $f\ge0$.} defined on
  $\bigcup_{1\le p\le \infty}L^p(X, \mathcal B,\mu)$ for some index set
  $Z\subseteq (0, \infty)$. Moreover, for a given $p_0\in[1, 2)$ the following conditions are
  satisfied:
\begin{itemize}
\item There is a family of linear operators $(S_n: n\in\ZZ)$ with the
property that $f=\sum_{n\in\ZZ}S_nf,$ for any
$f\in L^2(X, \mathcal B,\mu)$. Moreover,  for every $p\in(1, 2]$
there exists $C_{1, p}>0$ for which for every
$f\in L^p(X, \mathcal B,\mu)$ we have
  \begin{align}
    \label{eq:1}
       \bigg\|\Big(\sum_{n\in \ZZ}|S_n
    f|^2\Big)^{1/2}\bigg\|_{L^p}\leq C_{1, p}\|f\|_{L^p}.
  \end{align}  
  \item For every $p\in(p_0, 2]$ there exists
     $C_{2, p}>0$ such that for every $f\in L^p(X, \mathcal B,\mu)$ we have 
    \begin{align}
      \label{eq:11}
      \big\|\sup_{t\in Z}|T_tf|\big\|_{L^p}\le C_{2, p}\|f\|_{L^p}.
    \end{align}
    \item For every $p\in (p_0, 2]$ there exists $C_{3, p}>0$ such that for every
    $f\in L^p(X, \mathcal B,\mu)$ we have
      \begin{align}
        \label{eq:33}
       \sup_{n \in\ZZ}\big\|V_p\big(T_tf: t\in Z_n\big)\big\|_{L^p}\le C_{3, p}\|f\|_{L^p},  
      \end{align}
      where $Z_n=[2^n, 2^{n+1})\cap Z$.
    \item There exists a sequence $(a_j: j\in
      \ZZ)$ of positive numbers such that
      $\sum_{j\in\ZZ}a_j^{\rho}=A_{\rho}<\infty$ for every $\rho>0$ and there exists 
      $C_4>0$ such that  for every
      $j\in\ZZ$ and for every $f\in L^2(X, \mathcal B,\mu)$ we have
      \begin{align}
        \label{eq:92}
        \bigg\|\Big(\sum_{n\in \ZZ}V_2\big(T_tS_{j+n}
    f:t\in Z_n\big)^2\Big)^{1/2}\bigg\|_{L^2}\leq C_{4}a_j\|f\|_{L^2}.
      \end{align}
\end{itemize}
Then for every $p\in(p_0, 2]$, there exists $C_{p}>0$ such that for
every $f\in L^p(X, \mathcal B,\mu)$ we have
  \begin{align}
   \label{eq:72}
    \bigg\|\Big(\sum_{n\in \ZZ}V_2\big(T_t
    f:t\in Z_n\big)^2\Big)^{1/2}\bigg\|_{L^p}\le C_{p}\|f\|_{L^p}.
  \end{align}
\end{proposition}
\begin{proof}
Let us fix $p\in(p_0, 2)$. We will show that there is $C_p>0$
  such that for every $L, N\in\NN$ we have 
  \begin{align}
    \label{eq:73}
    \bigg\|\Big(\sum_{|n|\le N}V_2^L\big(T_t
    f:t\in Z_n\big)^2\Big)^{1/2}\bigg\|_{L^p}\le C_{p}\|f\|_{L^p}.
  \end{align}
  Letting $N\to\infty$ and $L\to\infty$ and invoking the monotone
  convergence theorem we obtain \eqref{eq:72}.
  
  To prove \eqref{eq:73}
  we reduce the problem to study a certain family of linear operators.
  For this purpose, for every $n\in\ZZ$, let $\mathfrak T_L^n$ be a
  family of all sequences $\mathfrak t^n_L=(t_l^n: 0\le l\le L)$ such
  that each component $t_l^n: \RR^d\to Z_n$ is a measurable function
  and $t_0^n(x)<t_1^n(x)<\ldots<t_L^n(x)$ for any $x\in\RR^d$.
We will prove that there is  $C_p>0$
  such that for every $L, N\in\NN$ we have 
  \begin{align}
\label{eq:23}
    \sup_{\mathfrak t_L^{-N}\in\mathfrak T_L^{-N}}\ldots\sup_{\mathfrak t_L^N\in\mathfrak T_L^N}\bigg\|\Big(\sum_{|n|\le N}\sum_{l=0}^{L-1}\big|(T_{t_{l+1}^n}-T_{t_{l}^n})
    f\big|^2\Big)^{1/2}\bigg\|_{L^p}\le C_{p}\|f\|_{L^p}.
  \end{align}
  Suppose for the moment that \eqref{eq:23} is proven. We will show how
  it implies \eqref{eq:73}. For every integer $|n|\le N$, let
  $\mathfrak t^n_{L, f}=(t_{l, f}^n: 0\le l\le L)$ be a sequence of
  measurable functions with the properties as above and additionally
  satisfying 
  \[
  V_2^L\big(T_t f(x):t\in Z_n\big)^2=\sum_{l=0}^{L-1}\big|(T_{t_{l+1,
      f}^n(x)}-T_{t_{l, f}^n(x)}) f(x)\big|^2.
  \]
  Then, assuming \eqref{eq:23}, we obtain
  \begin{align*}
    \bigg\|\Big(\sum_{|n|\le N}V_2^L\big(T_t
    f:t\in Z_n\big)^2\Big)^{1/2}\bigg\|_{L^p}&=\bigg\|\Big(\sum_{|n|\le N}\sum_{l=0}^{L-1}\big|(T_{t_{l+1,
                                               f}^n}-T_{t_{l, f}^n}) f\big|^2\Big)^{1/2}\bigg\|_{L^p}\\
    &\le\sup_{\mathfrak t_L^{-N}\in\mathfrak T_L^{-N}}\ldots\sup_{\mathfrak t_L^N\in\mathfrak T_L^N}\bigg\|\Big(\sum_{|n|\le N}\sum_{l=0}^{L-1}\big|(T_{t_{l+1}^n}-T_{t_{l}^n})
      f\big|^2\Big)^{1/2}\bigg\|_{L^p}\\
    &\le C_{p}\|f\|_{L^p},
  \end{align*}
  and \eqref{eq:73} follows. Thus it suffices to prove \eqref{eq:23}.
  
  Due to \eqref{eq:33} we see that \eqref{eq:23} holds with the
  constant $(2N+1)C_{3, p}$. Our task now will be to show that
  \eqref{eq:23} holds with the constant which is independent of
  $L, N\in\NN$.  For this purpose fix
  $\mathfrak t_L^{-N}\in\mathfrak T_L^{-N}, \ldots, \mathfrak
  t_L^{N}\in\mathfrak T_L^{N}$, and for $p\in(p_0, 2)$,
  $r\in[p, \infty]$ and $s\in[1, \infty]$, let $A_N(p, r, s)$ denote
  the best constant in the following inequality
\begin{align}
\label{eq:74}
  \bigg\lVert\bigg(\sum_{|n|\le N}\Big(\sum_{l=0}^{L-1}\big|(T_{t_{l+1}^n}-T_{t_{l}^n})
      g_n\big|^r\Big)^{s/r}
  \bigg)^{1/s}
  \bigg\rVert_{L^p}\le A_N(p, r, s)\bigg\|\Big(\sum_{|n|\le N}|g_n|^s\Big)^{1/s}\bigg\|_{L^p}.
\end{align}
Let $q$ be a real number such that  $p_0<q<p<2$ and define $\theta\in(0, 1)$ by setting
    \[
\frac{1}{2}=\frac{1-\theta}{q}+\frac{\theta}{\infty}.
      \]
      This implies that $\theta=1-q/2$ and consequently
      determines $u\in(q, p)$ such that
      \[
\frac{1}{u}=\frac{1-\theta}{q}+\frac{\theta}{p}.
\]
Therefore, interpolation
between $L^{q}\big(\ell^{q}(\ell^{q})\big)$ and
$L^{p}\big(\ell^{\infty}(\ell^{\infty})\big)$
ensures that
\[
A_N(u, 2, 2)\le A_N(q, q, q)^{1-\theta}A_N(p, \infty, \infty)^{\theta}.
\]
Invoking \eqref{eq:33} we have $A_N(q, q, q)\le C_{3, q}$, since
\begin{align*}
\bigg\lVert\Big(\sum_{|n|\le N}\sum_{l=0}^{L-1}\big|(T_{t_{l+1}^n}-T_{t_{l}^n})
      g_n\big|^q
  \Big)^{1/q}
  \bigg\rVert_{L^q}&\le 
  \bigg\lVert\Big(\sum_{|n|\le N}
  V_{q}\big(  T_t g_n: t\in  Z_n\big)^{q}\Big)^{1/{q}}
                     \bigg\rVert_{L^{q}}\\
  &\le C_{3, q}\bigg\|\Big(\sum_{|n|\le N}|g_n|^{q}\Big)^{1/{q}}\bigg\|_{L^{q}}.
\end{align*}
Let $g=\sup_{|n|\le N}|g_n|$ and observe that $A_N(p, \infty,
\infty)\le 2C_{2, p}$, since by \eqref{eq:11}, we obtain
\begin{align*}
\big\|\sup_{|n|\le N}\sup_{0\le l\le L-1}\big|(T_{t_{l+1}^n}-T_{t_{l}^n})
      g_n\big|\big\|_{L^p}\le2\big\|\sup_{t\in Z}T_tg
                   \big\|_{L^p}
  \le 2C_{2, p}\|g\|_{L^p}.
\end{align*}
Thus
\begin{align}
  \label{eq:75}
  A_N(u, 2, 2)\le A_N(q, q, q)^{1-\theta}A_N(p, \infty,
  \infty)^{\theta}\le
  C_{3, q}^{1-\theta}(2C_{2, p})^{\theta}.
\end{align}
By \eqref{eq:74}, \eqref{eq:1} and \eqref{eq:75} we get
\begin{align}
  \label{eq:76}
  \begin{split}
\bigg\lVert\Big(\sum_{|n|\le N}\sum_{l=0}^{L-1}\big|(T_{t_{l+1}^n}-T_{t_{l}^n})
      S_{j+n}f\big|^2
  \Big)^{1/2}
  \bigg\rVert_{L^u}
  &\le A_N(u, 2, 2)\bigg\|\Big(\sum_{|n|\le N}|S_{j+n}f
                                                      |^2\Big)^{1/2}\bigg\|_{L^u}\\
  &\le C_{1, u}C_{3, q}^{1-\theta}(2C_{2, p})^{\theta}\|f\|_{L^u}.
  \end{split}
\end{align}
By \eqref{eq:92} we get
\begin{align}
  \label{eq:25}
  \begin{split}
  \bigg\lVert\Big(\sum_{|n|\le N}\sum_{l=0}^{L-1}\big|(T_{t_{l+1}^n}-T_{t_{l}^n})
      S_{j+n}f\big|^2
  \Big)^{1/2}
  \bigg\rVert_{L^2}&\le
  \bigg\|\Big(\sum_{|n|\le N}V_2\big(T_t
  S_{j+n}f:t\in Z_n\big)^2\Big)^{1/2}\bigg\|_{L^2}\\
  &\le C_4a_j\|f\|_{L^2}.
  \end{split}
\end{align}
    Now let us introduce $\rho\in(0, 1)$ obeying
    \[
\frac{1}{p}=\frac{1-\rho}{u}+\frac{\rho}{2}.
      \]
Interpolating \eqref{eq:76} with \eqref{eq:25} we have
\begin{align}
  \label{eq:77}
  \bigg\lVert\Big(\sum_{|n|\le N}\sum_{l=0}^{L-1}\big|(T_{t_{l+1}^n}-T_{t_{l}^n})
      S_{j+n}f\big|^2
  \Big)^{1/2}
  \bigg\rVert_{L^p}\le \Big(C_{1, u}C_{3, q}^{1-\theta}(2C_{2, p})^{\theta}\Big)^{1-\rho}C_4^{\rho}a_j^{\rho}\|f\|_{L^p}.
\end{align}
Summing \eqref{eq:77} over $j\in\ZZ$ we obtain 
\begin{align*}
    \bigg\lVert\Big(\sum_{|n|\le N}\sum_{l=0}^{L-1}\big|(T_{t_{l+1}^n}-T_{t_{l}^n})
      f\big|^2
  \Big)^{1/2}
  \bigg\rVert_{L^p}\le \Big(C_{1, u}C_{3, q}^{1-\theta}(2C_{2, p})^{\theta}\Big)^{1-\rho}C_4^{\rho}A_{\rho}\|f\|_{L^p}.
\end{align*}
which proves \eqref{eq:23} and
 completes the proof of Proposition \ref{prop:2}.
\end{proof}

\subsection{The Poisson semigroup}
As in \cite{B1} we shall exploit dimension-free bounds for the
Poisson semigroup $P_t$ which,  for every $\xi \in \RR^d$, satisfies
  $$\calF (P_t f)(\xi)=p_t(\xi)\calF (f)(\xi),$$
  where
  \[
p_t(\xi)=e^{-2\pi t|\xi|}.
\]
For every $x\in\RR^d$ let us introduce the maximal function 
 $$P_{\star}f(x)=\sup_{t>0}P_t|f|(x)$$
 and the square function
$$g(f)(x)=\left(\int_{0}^{\infty}t\Big|\frac{\rm d}{{\rm d}t} P_t
  f(x)\Big|^2{\rm d}t\right)^{1/2}$$ associated with the Poisson
semigroup. We know from \cite{Ste1} that for every $p\in(1, \infty)$
there exists a constant  $C_p>0$ independent of the dimension such that
for every $f\in L^p(\RR^d)$ we have
\begin{align}
  \label{eq:78}
\|P_{\star}f\|_{L^p}\leq C_p\|f\|_{L^p}  
\end{align}
and
\begin{align}
  \label{eq:79}
  \|g(f)\|_{L^p}\leq C_p\|f\|_{L^p}.
\end{align}
We will also need some variant of the Littlewood--Paley
inequality. Namely, we define for every $n\in\ZZ$ the projections
  $S_n=P_{L 2^n}-P_{L 2^{n-1}}$ corresponding to the  Poisson
  semigroup, where $L=L(G)$ is the constant defined in \eqref{eq:68}.
 With this definition, we clearly have, for every $f\in L^2(\RR^d)$, that
 \begin{align}
   \label{eq:80}
f=\sum_{n\in \ZZ} S_n f.   
 \end{align}
 Moreover, for each $n\in\ZZ$ and $x\in\ZZ^d$ we see
 $$S_nf(x)=\int_{L2^{n-1}}^{L2^n}\frac{\rm d}{{\rm d}t} P_t f(x){\rm
   d}t.$$
  Hence, by the Cauchy--Schwarz inequality we obtain 
\begin{align*}
|S_nf(x)|^2\leq \bigg(\int_{L2^{n-1}}^{L2^n}\Big|\frac{\rm d}{{\rm d}t} P_t
  f(x)\Big|{\rm d}t\bigg)^2\le L2^{n-1}\int_{L2^{n-1}}^{L2^n}\Big|\frac{\rm d}{{\rm d}t} P_t
  f(x)\Big|^2{\rm d}t \le \int_{L2^{n-1}}^{L2^n}t\Big|\frac{\rm d}{{\rm d}t} P_t
  f(x)\Big|^2{\rm d}t.
\end{align*}
Now summing over $n\in\ZZ$ and appealing to \eqref{eq:79} we obtain that
the following dimension-free Littlewood--Paley inequality
\begin{align}
  \label{eq:81}
   \bigg\|\Big(\sum_{n\in \ZZ}|S_n f|^2\Big)^{1/2}\bigg\|_{L^p}\leq \|g(f)\|_{L^p}\le C_p\|f\|_{L^p},
 \end{align}
for every $f\in L^p(\RR^d)$.

 Finally, since the Poisson semigroup $P_t$ is Markovian,
 \cite[Theorem 3.3]{JR1} states that for every $p\in(1, \infty)$ and
 every $r\in(2, \infty)$ there is a constant $C_{p, r}>0$ independent
 of the dimension such that for every $f\in L^p(\RR^d)$ we have
\begin{equation}
\label{eq:2}
\big\|	V_r\big( P_{t} f: t>0\big)\big\|_{L^p}\leq C_{p, r}\|f\|_{L^p}.
\end{equation}
We include the proof of \eqref{eq:2} for completeness. It suffices to prove that
there is a constant $C_{p, r}>0$ such that for every $M\in\NN$ we have
\begin{equation}
\label{eq:28}
\big\|	V_r\big( P_{t} f: t\in \mathcal D_M\big)\big\|_{L^p}\leq C_{p, r}\|f\|_{L^p},
\end{equation}
where $\mathcal D_M=\{n/2^M: n\ge0\}$. Then letting $M\to\infty$ in \eqref{eq:28}, by the
monotone convergence theorem, we obtain  \eqref{eq:2}.
To prove \eqref{eq:28} we recall 
Rota's theorem from \cite{Ste1}.

\begin{theorem}[Rota's theorem]
Assume that $(X, \mathcal B, \mu)$ is a $\sigma$-finite measure space
and let $Q$ be a linear operator defined on
$\bigcup_{1\le p\le \infty}L^p(X, \mathcal B, \mu)$ satisfying the
following conditions:
\begin{itemize}
\item $\|Q\|_{L^p\to L^p}\le 1$ for every $p\in[1, \infty]$,
\item $Q=Q^*$ in $L^2(X, \mathcal B, \mu)$,
\item $Qf\ge0$ for every $f\ge0$,
\item $Q1=1$.
\end{itemize}
Then there exists a measure space $(\Omega, \mathcal F, \nu)$, a decreasing collection of $\sigma$-algebras
$
\ldots\subset  \mathcal F_{n+1}\subset \mathcal F_{n}\subset \ldots \subset \mathcal F_{1}\subset \mathcal F_{0}= \mathcal F
$
and another $\sigma$-algebra $\hat {\mathcal F}\subset \mathcal F$ with the following properties:
\begin{itemize}
\item[(i)] There exists a measure space isomorphism
$i: (X, \mathcal B, \mu)\to (\Omega, \hat{\mathcal F}, \nu)$,
which induces an isomorphism
$
\mathfrak i: L^p(X, \mathcal B, \mu)\to L^p(\Omega,\hat{\mathcal F},  \nu)
$
given by the formula
$\mathfrak i(f)(\omega)=f(i^{-1}\omega)$ for every $\omega\in\Omega$.
\item[(ii)] For every $f\in L^p(X,  \mathcal B, \mu)$ and for every $x\in X$ we have
\begin{align}
  \label{eq:27}
Q^{2n}(f)(x)=  Q^{2n}(\mathfrak i^{-1}F)(x)=\hat{\mathbb E}\big[\mathbb E[F|\mathcal F_n]\big|\hat{\mathcal F}\big](i(x)),
\end{align}
where $F=\mathfrak i (f)\in  L^p(\Omega,\hat{\mathcal F},  \nu)$.
\end{itemize}
Thus the operator $Q$ is associated with a family of reverse martingales $(\mathbb E[F|\mathcal F_n]: n\ge0)$.
\end{theorem}
To prove \eqref{eq:28} we apply Rota's theorem with $Q=P_{1/2^{M+1}}$
and with $X=\RR^d$ and the $\sigma$-algebra $\mathcal B$ of all
Lebesgue measurable sets on $\RR^d$ and the Lebesgue measure $\mu=|\cdot|$
on $\RR^d$.
Indeed,
\begin{align*}
  \big\|	V_r\big( P_{t} f: t\in \mathcal D_M\big)\big\|_{L^p}
&  =\big\|	V_r\big( P_{n/2^M} f: n\ge0\big)\big\|_{L^p}
  =\big\|V_r\big( Q^{2n} f: n\ge0\big)\big\|_{L^p}\\
 & =\big\|V_r\big( \hat{\mathbb E}\big[\mathbb E[F|\mathcal F_n]\big|\hat{\mathcal F}\big]: n\ge0\big)\big\|_{L^p(\Omega,\hat{\mathcal F},\nu)}\\
&  \le\big\|\hat{\mathbb E}\big[V_r\big( \mathbb E[F|\mathcal F_n]: n\ge0\big)\big|\hat{\mathcal F}\big]\big\|_{L^p(\Omega,\hat{\mathcal F},\nu)}\\
& \le\big\|V_r\big( \mathbb E[F|\mathcal F_n]: n\ge0\big)\big\|_{L^p(\Omega,{\mathcal F},\nu)}\\
&  \leq C_{p, r}\|F\|_{L^p(\Omega,\mathcal F, \nu)}=C_{p, r}\|f\|_{L^p(X,\mu)},  
\end{align*}
where the last inequality follows from L\'epingle's inequality for martingales, see \cite{JSW} and the references given there.
As there is no 'dimension' in the proof of L\'epingle's inequality (it is purely
probabilistic) the implied constant $C_{p, r}$ in \eqref{eq:2} is
independent of the dimension.

\section{Estimates for long variations: proof of Theorem \ref{thm:3}}
\label{sec:3}
In this section we prove Theorem \ref{thm:3}. Fix $r\in(2, \infty)$
and a non-empty convex symmetric body $G$ in $\RR^d$. Recall that $M_t=M_t^G$ and $m=m^G$, where
$m^G(\xi)=\mathcal F(\ind{G})(\xi)$. The main ingredients in the proof
will be the $r$-variational estimate \eqref{eq:2} for the Poisson
semigroup $P_t$, the estimates for the multiplier $m$ associated with
$M_t$ from Proposition \ref{prop:1}, and the lacunary maximal result
from \cite{B2} or \cite{Car1}, which states that for every
$p\in(1, \infty]$ there exists $C_{p, \infty}>0$ such that every
$d\in\NN$ and for every convex body $G\subset \RR^d$ the following
inequality holds
\begin{align}
  \label{eq:82}
  \big\|\sup_{n\in\ZZ}|M_{2^n}f|\big\|_{L^p}\le C_{p, \infty}\|f\|_{L^p}
\end{align}
for all $f\in L^p(\RR^d)$.

\begin{proof}[Proof of Theorem \ref{thm:3}]

For every $f\in L^p(\RR^d)$  we obtain 
\begin{align}
\label{eq:83}
  \big\lVert
	V_r\big(  M_{ 2^n} f: n\in\ZZ\big)
	\big\rVert_{L^p}\le
  \big\lVert
	V_r\big( P_{L 2^n} f: n\in\ZZ\big)
	\big\rVert_{L^p}
  +\bigg\|\Big(\sum_{n\in\ZZ}\big| M_{2^n} f- 
 P_{L 2^n}f\big|^2\Big)^{1/2}\bigg\|_{L^p}.
\end{align}
The first term in \eqref{eq:83} is bounded on $L^p$ by
\eqref{eq:2}. Therefore, it remains to obtain $L^p$ bounds for the
square function in \eqref{eq:83}. For this purpose we will use
\eqref{eq:80}. Indeed, observe that
\begin{align}
  \label{eq:87}
  \begin{split}
\bigg\|\Big(\sum_{n\in\ZZ}\big| M_{2^n} f&- 
  P_{L2^n}f\big|^2\Big)^{1/2}\bigg\|_{L^p}
 =\bigg\|\Big(\sum_{n\in\ZZ}\big|\sum_{j\in \ZZ}  M_{2^n}S_{j+n}f - 
 P_{L2^n}S_{j+n}f\big|^2\Big)^{1/2}\bigg\|_{L^p}\\
 &\le \sum_{j\in \ZZ}\bigg\|\Big(\sum_{n\in\ZZ}\big|  M_{2^n} S_{j+n}f - 
 P_{L2^n}S_{j+n}f\big|^2\Big)^{1/2}\bigg\|_{L^p}\lesssim
 \sum_{j\in\ZZ}2^{-\delta_p|j|}\|f\|_{L^p}\lesssim \|f\|_{L^p}.    
  \end{split}
\end{align}
In order to justify the last but one inequality in \eqref{eq:87} we
shall prove, for each $j\in\ZZ$, that
\begin{equation}
\label{eq:13} 
\bigg\|\Big(\sum_{n\in\ZZ}\big|  M_{2^n} S_{j+n}f \big|^2\Big)^{1/2}\bigg\|_{L^p}+\bigg\|\Big(\sum_{n\in\ZZ}\big|
P_{L2^n}S_{j+n}f\big|^2\Big)^{1/2}\bigg\|_{L^p}\lesssim\|f\|_{L^p}.
\end{equation}
and
\begin{equation}
\label{eq:12} 
\bigg\|\Big(\sum_{n\in\ZZ}\big|  M_{2^n} S_{j+n}f - 
P_{L2^n}S_{j+n}f\big|^2\Big)^{1/2}\bigg\|_{L^2}\lesssim2^{-|j|/2}\|f\|_{L^2}.
\end{equation}
To prove \eqref{eq:13} we first show  the following dimension-free vector-valued bounds
\begin{align}
  \label{eq:29}
  \bigg\|\Big(\sum_{n\in\ZZ}\big|M_{2^n} g_n\big|^2\Big)^{1/2}\bigg\|_{L^p}\lesssim
  \bigg\|\Big(\sum_{n\in\ZZ}|g_n|^2\Big)^{1/2}\bigg\|_{L^p}
\end{align}
and
\begin{align}
\label{eq:30}
  \bigg\|\Big(\sum_{n\in\ZZ}\big|P_{L2^n} g_n\big|^2\Big)^{1/2}\bigg\|_{L^p}\lesssim
  \bigg\|\Big(\sum_{n\in\ZZ}|g_n|^2\Big)^{1/2}\bigg\|_{L^p}
\end{align}
for all $p\in(1, \infty)$. Then in view of \eqref{eq:29}, \eqref{eq:30} and \eqref{eq:81} we conclude
\begin{align*}
\bigg\|\Big(\sum_{n\in\ZZ}\big|  M_{2^n} S_{j+n}f \big|^2\Big)^{1/2}\bigg\|_{L^p}+\bigg\|\Big(\sum_{n\in\ZZ}\big|
P_{L2^n}S_{j+n}f\big|^2\Big)^{1/2}\bigg\|_{L^p}\lesssim\bigg\|\Big(\sum_{n\in\ZZ}\big|
S_{j+n}f\big|^2\Big)^{1/2}\bigg\|_{L^p}\lesssim\|f\|_{L^p},
\end{align*}
which proves \eqref{eq:13}.

The proof of \eqref{eq:29} and
\eqref{eq:30} follows respectively from \eqref{eq:82} and
\eqref{eq:78} and a vector-valued interpolation.  We only estimate
\eqref{eq:29}, the estimate in \eqref{eq:30} will be obtained
similarly. Indeed, for $p\in(1, \infty)$ and
$s\in[1, \infty]$, let $A(p, s)$ be the best constant in the following
inequality
\begin{align*}
  \bigg\|\Big(\sum_{n\in\ZZ}\big|M_{2^n} g_n\big|^s\Big)^{1/s}\bigg\|_{L^p}\le A(p, s)
  \bigg\|\Big(\sum_{n\in\ZZ}|g_n|^s\Big)^{1/s}\bigg\|_{L^p}.
\end{align*}
Then interpolation, duality ($A(p, s)=A(p', s')$), and \eqref{eq:82}  yield \eqref{eq:29}, since
\[
A(p, 2)\le A(p, 1)^{1/2}A(p, \infty)^{1/2}=A(p', \infty)^{1/2}A(p, \infty)^{1/2}\le C_{p', \infty}^{1/2}C_{p, \infty}^{1/2}.
\]

To prove \eqref{eq:12} let us introduce
$k(\xi)=m(\xi)-p_L(\xi)=m(\xi)-e^{-2\pi L|\xi|}$
    which is the multiplier associated with the operator $M_1-P_L$.
    Observe that by Proposition \ref{prop:1} and the properties of $p_L(\xi)$
    there exists a constant $C>0$ independent of the dimension such
    that
  \begin{align}
    \label{eq:88}
    |k(\xi)|\le |m(\xi)-1|+|p_L(\xi)-1|\le CL|\xi|, \qquad |k(\xi)|\le
    C(L|\xi|)^{-1}, \qquad |\langle\xi, \nabla k(\xi)\rangle|\le C,
  \end{align}
since $\langle\xi, \nabla p_L(\xi)\rangle=-2\pi L|\xi|e^{-2\pi
  L|\xi|}$. Therefore, by \eqref{eq:88} and Plancherel's theorem we
get
\begin{multline}
  \label{eq:71}
\bigg\|\Big(\sum_{n\in\ZZ}\big|  M_{2^n} S_{j+n}f - 
P_{L2^n}S_{j+n}f\big|^2\Big)^{1/2}\bigg\|_{L^2}\\
  =\bigg(\int_{\RR^d}\sum_{n\in
  \ZZ}\big|k(2^n\xi)\big(e^{-2\pi2^{n+j} L|\xi|}
  -e^{-2\pi2^{n+j-1} L|\xi|}\big)\big|^2
  |\mathcal Ff(\xi)|^2{\rm d}\xi\bigg)^{1/2}\\
  \lesssim\bigg(\int_{\RR^d}\sum_{n\in
  \ZZ}\min\big\{2^nL|\xi|, (2^nL|\xi|)^{-1}\big\}^2\big|\big(e^{-2\pi2^{n+j} L|\xi|}
  -e^{-2\pi2^{n+j-1} L|\xi|}\big)\big|^2
  |\mathcal Ff(\xi)|^2{\rm d}\xi\bigg)^{1/2}\\
  \lesssim2^{-|j|/2}\bigg(\int_{\RR^d}\sum_{n\in
  \ZZ}\min\big\{2^nL|\xi|, (2^nL|\xi|)^{-1}\big\}
  |\mathcal Ff(\xi)|^2{\rm d}\xi\bigg)^{1/2}
  \lesssim 2^{-|j|/2}\|f\|_{L^2},
\end{multline}
where we have used
\begin{align}
  \label{eq:89}
  \min\big\{2^nL|\xi|, (2^nL|\xi|)^{-1}\big\}\big|\big(e^{-2\pi2^{n+j} L|\xi|}
  -e^{-2\pi2^{n+j-1} L|\xi|}\big)\big|\lesssim 2^{-|j|}
\end{align}
and
\begin{align}
  \label{eq:90}
  \sum_{n\in
  \ZZ}\min\big\{2^nL|\xi|, (2^nL|\xi|)^{-1}\big\}\lesssim 1.
\end{align}
The proof of Theorem \ref{thm:3} is completed.
\end{proof}

\section{Estimates for short variations: proofs of Theorem
  \ref{thm:6} and Theorem \ref{thm:7}}
This section is devoted to prove Theorem
\ref{thm:6}, and Theorem \ref{thm:7}.
All these theorems will follow, in  view
  of \eqref{eq:21} and Theorem \ref{thm:3} which has been proven in
  the previous section, if we show that  for appropriate
parameters  $p$
the short variation seminorm is bounded, i.e.  for every function
$f\in L^p(\RR^d)$ the following inequality
\begin{align}
  \label{eq:91}
     \bigg\|\Big(\sum_{n\in\ZZ}V_2\big(M_{t}f: t\in[2^n,
   2^{n+1})\big)^2\Big)^{1/2}\bigg\|_{L^p}\le C_{p}\|f\|_{L^p}
\end{align}
holds
 with a constant $C_{p}>0$ which
does not depend on the dimension, where $M_t=M_t^G$ as in the previous section.

One of the  key tools in all of the proofs will be inequality \eqref{eq:20}
from Lemma \ref{lem:6}. This inequality and the Littlewood--Paley
decomposition \eqref{eq:80} will result in \eqref{eq:91} for $p\ge2$. The proof of
\eqref{eq:91} for $p\le2$ will additionally require an almost orthogonality
principle from Proposition \ref{prop:2}. Due to \eqref{eq:81} and the
maximal results from \cite{B2}, \cite{Car1}, \cite{Mul1} and \cite{B3}
we will have to only verify conditions \eqref{eq:33} and \eqref{eq:92}
from Proposition \ref{prop:2}.  The bound \eqref{eq:92} will
follow from inequality \eqref{eq:20}, the Littlewood--Paley
decomposition \eqref{eq:80}, and properties of the multiplier $m^G$
from Proposition \ref{prop:1}.  The bounds for \eqref{eq:33} will
require a more sophisticated argument, which is subsumed in Lemma
\ref{lem:7}. 

Now we need to recall some facts from \cite{Car1}, \cite{Mul1}, and \cite{B3}.
For $\alpha\in(0, 1)$ let $\calD^{\alpha}$ be the fractional derivative
\begin{equation}
\label{eq:104}
\calD^{\alpha}F(t)=\calD^{\alpha}_tF(t)=\calD^{\alpha}_uF(u)\big|_{u=t}=\calF_{\RR}\big((2i\pi \xi )^{\alpha}\calF_{\RR}^{-1}(F)(\xi)\big)(t) 
\end{equation}
for every $t\in \RR$. This formula gives a well defined tempered
distribution on $\RR.$ Note the resemblance of the fractional
derivative $\mathcal D^{\alpha}$ with its variant $D^{\alpha}$ used in
Proposition \ref{prop:chVr}. In fact, the version of that proposition,
with $\mathcal D^{\alpha}$ in place of $D^{\alpha}$, holds as well. Since
we do not explicitly need it, we will forgo a proof of this fact.

Simple computations show that for $t>0$ we have
\begin{align*}
\calD_t^{\alpha}m(t\xi)=\int_{\RR^d}(2\pi i
x\cdot\xi)^{\alpha}K(x)e^{-2\pi i tx\cdot\xi}{\rm d}x,
\end{align*}
where $m=m^G=\mathcal F(K_G)$, and $K(x)=K_G(x)=\ind{G}(x)$.
Moreover, \cite[Lemma 6.6]{DGM1} guarantees that
\begin{align}
  \label{eq:40}
\calD_t^{\alpha}m(t\xi)
=-\frac1{\Gamma(1-\alpha)}\int_t^{\infty}(u-t)^{-\alpha}\frac{{\rm d}}{{\rm d}u}m(u\xi){\rm d}u.
\end{align}
If $\mathcal P_u^{\alpha}$ is the operator associated with the multiplier
\[
\mathfrak p_u^{\alpha}(\xi)=u^{\alpha+1}\calD_v^{\alpha}\bigg(\frac{m(v\xi)}{v}\bigg)\bigg|_{v=u}
\]
for $\xi\in\RR^d$, then one can see that
\begin{align}
\label{eq:99}
M_tf(x)=\mathcal F^{-1}(m(t\xi)\mathcal F
f)(x)=\frac{1}{\Gamma(\alpha)}\int_t^{\infty}\frac{t}{u}
\bigg(1-\frac{t}{u}\bigg)^{\alpha-1}\mathcal P_{u}^{\alpha}f(x)\frac{{\rm d}u}{u}.
\end{align}
Carbery \cite{Car1} showed that for general symmetric convex bodies one has
\begin{align}
  \label{eq:34}
  \|\mathcal P_1^{\alpha} f\|_{L^p}\lesssim \|f\|_{L^p}+\|T_{(\xi\cdot\nabla)^{\alpha}m} f\|_{L^p},
\end{align}
where $T_{(\xi\cdot\nabla)^{\alpha}m} f$ is  the multiplier operator associated with the symbol
\[
(\xi\cdot\nabla)^{\alpha}m(\xi)=\calD_t^{\alpha}m(t\xi)|_{t=1}.
\]
The estimate from \eqref{eq:34} immediately implies that
\begin{align}
  \label{eq:38}
  \sup_{u>0}\|\mathcal P_u^{\alpha}\|_{L^p\to L^p}\lesssim_p 1+\|T_{(\xi\cdot\nabla)^{\alpha}m}\|_{L^p\to L^p},
\end{align}
since the
multipliers $\mathfrak p_u^{\alpha}$ are dilations of $\mathfrak p_1^{\alpha}$, i.e. $\mathfrak p_u^{\alpha}(\lambda\xi)=\mathfrak p_{\lambda u}^{\alpha}(\xi)$ for any $\lambda>0$. 
Using identity \eqref{eq:99} it follows from \cite{Car1}  that for a general symmetric convex body $G$, every
$p\in(3/2, \infty]$, and for every $f\in L^p(\RR^d)$ we have
\begin{align}
  \label{eq:32}
    \big\|\sup_{t>0}|M_{t}f|\big\|_{L^p}\lesssim_p\|f\|_{L^p},
\end{align}
with the implicit constant independent of the dimension and the
underlying body. The proof of \eqref{eq:32} consists of two steps.  In
the first step it was proven that the $L^p$ bundedness in
\eqref{eq:32} can be deduced from \eqref{eq:38} provided that
$\alpha>1/p$. Then using complex interpolation Carbery showed that there is $C_{\alpha, p}>0$ such that
\begin{align}
  \label{eq:54}
\|T_{(\xi\cdot\nabla)^{\alpha}m}\|_{L^p\to L^p}\lesssim_p C_{\alpha, p}  
\end{align}
for $\alpha=2-2/p$. Combining these two facts we obtain $1/p<2-2/p$
which is equivalent to $p>3/2$ and gives the desired range
in \eqref{eq:32}.

When $G=B_q$ is as in \eqref{eq:66} with $q\in[1,\infty)$, in
\cite{Mul1} it has been proven that for every $\alpha\in(1/2, 1)$ and
every $p\in(1, \infty)$ there is a constant $C_{\alpha, p, q}>0$
independent of $d$ such that
\begin{align}
\label{eq:98'}
\|T_{(\xi\cdot\nabla)^{\alpha}m} \|_{L^p\to L^p}\le C_{\alpha, p, q}.
\end{align}
The same estimate for
$G=B_{\infty}$ is justified in \cite{B3}. In fact, in both \cite{Mul1}
and \cite{B3} the estimate \eqref{eq:98'} boils down to controlling
the operator $T_{|\xi|m(\xi)}$ associated with the multiplier
$|\xi|m(\xi).$ In view of \eqref{eq:34} and \eqref{eq:98'},  for the
bodies studied in Theorem \ref{thm:7} we thus have for all
$p\in(1, \infty)$ and all $f\in L^p(\RR^d)$ that
\begin{align}
\label{eq:98}
\sup_{u>0}\big\|\mathcal P_u^{\alpha}f\big\|_{L^p}\lesssim_{\alpha, p}\|f\|_{L^p}.
\end{align}

Let $\eta$ be a smooth function on $\RR$ such that $0\le \eta(t)\le1$ and
\begin{align}
  \label{eq:3}
  \begin{split}
  \eta(t)=
  \begin{cases}
  1, & \text{ if } t\in [1,2],\\
  0, & \text{ if } t\not\in(1/2, 3).
  \end{cases}
  \end{split}
\end{align}

\begin{lemma}
\label{lem:7}
Let $\eta$ be a smooth function as in \eqref{eq:3}.
Then for  any $p\in(1, 2)$ and any $\alpha\in(1/p, 1)$ there is
$C_{\alpha, p}>0$ such that  for every
Schwartz function $f\in\mathcal S(\RR^d)$ we have
\begin{align}
  \label{eq:24}
  \big\|V_p\big((\eta(t)M_tf):t\in\RR\big)\big\|_{L^p(\RR^d)}\le C_{\alpha, p} \big(\|f\|_{L^p}+\big\|\eta(t)\mathcal F^{-1}\big(\mathcal D_t^{\alpha}m(t\xi)\mathcal Ff(\xi)\big)\big\|_{L^p(\RR\times\RR^d)}\big).
\end{align}
Moreover, we have
\begin{align}
\label{eq:4}
  \big\|V_p\big((\eta(t)M_tf):t\in\RR\big)\big\|_{L^p(\RR^d)}
  \le C_{\alpha, p} \big(\|f\|_{L^p}+
  \sup_{t>0}\big\|\mathcal F^{-1}\big({((t\xi) \cdot \nabla)^{\alpha}m(t\xi)}\mathcal Ff(\xi)\big)\big\|_{L^p}\big).
\end{align}
\end{lemma}
\begin{proof}
In the proof we abbreviate $F(t,x)=\eta(t)M_{t}f(x).$ Note that Proposition
\ref{prop:chVr} and Fubini's theorem give, for $\alpha>1/p$, that
\begin{equation}
\label{eq:5}
\big\|V_p\big(  F(t,\cdot) : t\in \RR\big)\big\|_{L^p}\lesssim \|F\|_{L^p(\RR\times \RR^d)}+\|D^{\alpha}_tF\|_{L^p(\RR\times \RR^d)}.
\end{equation}
For every $\alpha\in(0, 1)$ and $s\in\RR$ we have
\[
(2\pi i s)^{\alpha}=|2\pi s|^{\alpha}e^{i\frac{\pi\alpha}{2}{\rm
    sgn}(s)},\qquad\textrm{or, equivalently,}\qquad|2\pi s|^{\alpha}=(2\pi i
s)^{\alpha}e^{-i\frac{\pi\alpha}{2}{\rm sgn}(s)}.
\]
Taking into account this identity and  the $L^p(\RR)$ boundedness of the projections
\[
\Pi_{\pm}(h)=\calF_{\RR}\big(\ind{(0,\infty)}(\pm\xi)\calF^{-1}_{\RR}h(\xi)\big)
\]
we note that
\[
\|D_t^{\alpha}h\|_{L^p(\RR)}=\big\|e^{i\frac{\pi\alpha}{2}}\Pi_{-}\calD_t^{\alpha}h+e^{-i\frac{\pi\alpha}{2}}\Pi_{+}\calD_t^{\alpha}h\big\|_{L^p(\RR)}
\lesssim\|\calD_t^{\alpha}h\|_{L^p(\RR)}.
\]
Thus, by \eqref{eq:5} and Fubini's theorem we obtain
 \begin{equation*}
 \big\|V_p\big(  F(t,\cdot) : t\in \RR\big)\big\|_{L^p}\lesssim \|F\|_{L^p(\RR\times \RR^d)}+\|\calD^{\alpha}_tF\|_{L^p(\RR\times \RR^d)}.
 \end{equation*}
 
 We claim that
 \begin{equation}
 \label{eq:103}
\|\calD^{\alpha}_tF\|_{L^p(\RR\times \RR^d)}\lesssim \|f\|_{L^p}+\|\eta\, \calD^{\alpha}_t M_tf\|_{L^p(\RR\times\RR^d)}.
\end{equation}
To prove \eqref{eq:103} we have to establish \eqref{eq:6}. Suppose that $h$ is a function in $\mathcal C^{2}(\RR)$ such that
\begin{align}
  \label{eq:7}
  \bigg|\bigg(\frac{{\rm d}}{{\rm d}t}\bigg)^jh(t)\bigg|\lesssim (1+|t|)^{-j-1},
\end{align}
for $j=0, 1, 2$, and
\begin{align}
  \label{eq:8}
    \sup_{s\in\RR}|(1+s^2)\calF_{\RR}^{-1}h(s)|\lesssim1.
\end{align}
Then by \cite[Lemma 6.6]{DGM1} when $\alpha\in(0, 1)$ we get
\begin{align*}
  \calD_t^{\alpha}h(t)&=-\frac1{\Gamma(1-\alpha)}\int_0^{\infty}u^{-\alpha}\frac{{\rm d}}{{\rm d}u} h(t+u){\rm d}u\\
 & =-\frac1{\Gamma(1-\alpha)}\int_0^{\infty}u^{-\alpha}\frac{{\rm d}}{{\rm d}u} \big(h(t+u)-h(t\big)){\rm d}u.
\end{align*}
So an integration by parts yields
\begin{align}
  \label{eq:6}
  \calD_t^{\alpha}h(t)=-\frac{\alpha}{\Gamma(1-\alpha)}\lim_{\varepsilon\to0}\int_{\varepsilon}^{\infty}u^{-\alpha-1} \big(h(t+u)-h(t)\big){\rm d}u.
\end{align}
We fix $x\in\RR^d$ and take alternatively $h(t)=\eta(t)M_tf(x)$, which
is a Schwartz function or $h(t)=M_tf(x)$ which is a function in
$\mathcal C^{\infty}(\RR)$, since we have assumed that
$f\in\mathcal S(\RR^d)$. To be able to apply formula \eqref{eq:6} with
these functions we have to only verify \eqref{eq:7} and \eqref{eq:8}. Indeed, for $h(t)=\eta(t)M_tf(x)$ or $h(t)=M_tf(x)$ we get
$$\bigg|\frac{{{\rm d}^{j}}}{{\rm d}t^j}h(t)\bigg|\le C(f,x,d)\, (1+|t|)^{-j-1}.$$
  We also have
$$\sup_{s\in\RR}|(1+s^2)\calF_{\RR}h(s)|\leq C(f,x,d).$$ 
Although the last two inequalities
have bounds which depend on the dimension, it does not affect our
result. These inequalities are only used to check that we are allowed
to apply \cite[Lemma 6.6]{DGM1} to establish \eqref{eq:6}, which
itself is independent of the dimension.

Using \eqref{eq:6} we see that
\begin{align}
  \label{eq:9}
  \calD^{\alpha}_tF(t, x)-\eta(t)\calD^{\alpha}_tM_tf(x)=E(t,x)=\int_0^{\infty}K(t, u)M_{t+u}f(x){\rm d}u,
\end{align}
with
\[
K(t, u)=-\frac{\alpha}{\Gamma(1-\alpha)}\frac{\eta(t+u)-\eta(t)}{u^{\alpha+1}}.
\]
Now for $K(t, u)$ we have the following two immediate estimates
\begin{align}
  \label{eq:14}
  |K(t, u)|\lesssim \min\big\{u^{-\alpha}, u^{-\alpha-1}\big\}.
\end{align}
Observe now that
\begin{align*}
\|E\|_{L^p(\RR\times\RR^d)}\le \|E\|_{L^p((-\infty, 3)\times\RR^d)}+\|E\|_{L^p((3, \infty)\times\RR^d)}+\|E\|_{L^p([-3,3]\times\RR^d)}.  
\end{align*}
We estimate each term separately. For the first one, if $t<-3$ then
$\eta(t)=0$ and $\eta(t+u)\not=0$ only when $1/2<t+u<3$. Thus by Minkowski's integral inequality we obtain
\begin{align*}
  \|E\|_{L^p((-\infty, 3)\times\RR^d)}&\le\bigg(\int_{-\infty}^{-3}\Big(\int_0^{\infty}|K(t, u)|\sup_{t>0}\|M_tf\|_{L^p}{\rm d}u\Big)^p{\rm d}t\bigg)^{1/p}\\
                                      &  \le\sup_{t>0}\|M_tf\|_{L^p}\bigg(\int_{-\infty}^{-3}\Big(\int_{1/2-t}^{3-t}u^{-\alpha-1}{\rm d}u\Big)^p{\rm d}t\bigg)^{1/p}\\
                                      &  \lesssim\|f\|_{L^p}\bigg(\int_{-\infty}^{-3}|t|^{-p(\alpha+1)}{\rm d}t\bigg)^{1/p}\lesssim\|f\|_{L^p}.
\end{align*}
For the second one we have $\|E\|_{L^p((3, \infty)\times\RR^d)}=0$, since $\eta(t+u)=\eta(t)=0$ for $t>3$.
Finally, by \eqref{eq:14} and Minkowski's integral inequality we have
\begin{align*}
  \|E\|_{L^p([-3,3]\times\RR^d)}\lesssim &\bigg(\int_{-3}^{3}\Big(\int_0^{\infty}\min\big\{u^{-\alpha}, u^{-\alpha-1}\big\}\sup_{t>0}\|M_tf\|_{L^p}{\rm d}u\Big)^p{\rm d}t\bigg)^{1/p}\\
  \lesssim &\sup_{t>0}\|M_tf\|_{L^p}\lesssim \|f\|_{L^p}.
\end{align*}
Therefore, we have proven that
\begin{align}
\label{eq:16}
  \|E\|_{L^p(\RR\times\RR^d)}\lesssim\|f\|_{L^p},
\end{align}
which in view of \eqref{eq:9} yields \eqref{eq:103}.

Now, by Fubini's theorem applied twice we have
\begin{align*}
  \calD^{\alpha}_tM_tf(x)&=-\frac1{\Gamma(1-\alpha)}\int_t^{\infty}(u-t)^{-\alpha}\frac{{\rm d}}{{\rm d}u}M_uf(x){\rm d}u\\
                         &=-\frac1{\Gamma(1-\alpha)}\int_t^{\infty}(u-t)^{-\alpha}\mathcal F^{-1}\bigg(\frac{{\rm d}}{{\rm d}u}m(u\xi)\mathcal Ff(\xi)\bigg)(x){\rm d}u\\
                         &=\mathcal F^{-1}\bigg(\bigg(-\frac1{\Gamma(1-\alpha)}\int_t^{\infty}(u-t)^{-\alpha}\frac{{\rm d}}{{\rm d}u}m(u\xi){\rm d}u\bigg) \mathcal Ff(\xi)\bigg)(x)\\
  &=\mathcal F^{-1}\big(\mathcal D_t^{\alpha}m(t\xi)\mathcal Ff(\xi)\big)(x).
\end{align*}
This combined with \eqref{eq:103} gives \eqref{eq:24}. Finally by formula \eqref{eq:40} we obtain for any $\xi\in \RR^d$ and $\lambda>0$ that
  \[
\mathcal D_s^{\alpha}m(s\lambda\xi)\big|_{s=t}=\lambda^{\alpha}\mathcal D_s^{\alpha}m(s\xi)\big|_{s=\lambda t},
\]
thus
 \[
\mathcal D_s^{\alpha}m(s\xi)\big|_{s=t}=t^{-\alpha}\mathcal D_s^{\alpha}m(s(t\xi))\big|_{s=1}=t^{-\alpha}\big((t\xi)\cdot\nabla\big)^{\alpha}m(t\xi).
\]
This identity combined with \eqref{eq:24} implies
\begin{align*}
\big\|\eta(t)\mathcal F^{-1}\big(\mathcal D_t^{\alpha}m(t\xi)\mathcal Ff(\xi)\big)\big\|_{L^p(\RR\times\RR^d)}&\le
  \sup_{t>0}\big\|\mathcal F^{-1}\big(\big((t\xi)\cdot\nabla\big)^{\alpha}m(t\xi)\mathcal Ff(\xi)\big)\big\|_{L^p},
\end{align*}
which proves \eqref{eq:4} and completes the proof of the lemma.
\end{proof}

\subsection{Proof of Theorem \ref{thm:6}}
Fix a non-empty convex symmetric body $G$ in $\RR^d$. In view of Theorem \ref{thm:3} we are left with 
proving \eqref{eq:91} for all $p\in(3/2, 4)$. For simplicity of the notation as in the previous
section we will write $M_t=M_t^G$ and $m=m^G$.
In the proof we will use  the maximal result
from \cite{B2} or \cite{Car1}, which states that for every
$p\in(3/2, \infty]$ there exists $C_{p, \infty}>0$ such that every
$d\in\NN$ and for every convex body $G\subset \RR^d$ the following
inequality holds
\begin{align}
  \label{eq:94}
  \big\|\sup_{t>0}|M_{t}f|\big\|_{L^p}\le C_{p, \infty}\|f\|_{L^p}
\end{align}
for all $f\in L^p(\RR^d)$. We will also appeal to the lacunary maximal inequality from \eqref{eq:82}.

The proof will be split according to whether $p\in(2,4)$ or
$p\in (3/2,2].$ In both cases we shall exploit the $L^2$
inequality
            \begin{align}
\label{eq:56}
\bigg\|\Big(\sum_{n\in\ZZ} \sum_{k = 0}^{2^{l}-1} \big|M_{2^n+{2^{n-l}(k+1)}}S_{j+n}f
  - M_{2^n+{2^{n-l}k}}S_{j+n}f\big|^2
  \Big)^{1/2}\bigg\|_{L^2}
  \lesssim 2^{-(1-\varepsilon)l/2}2^{-\varepsilon|j|/4}\|f\|_{L^2},
\end{align}
valid for all $\varepsilon\in(0, 1)$, $j\in \ZZ$ and $l\ge0$, with
the implicit constant which does not depend on $j$ and $l$.

\begin{proof}[Proof of the estimate \eqref{eq:56}] We will need some preparatory estimates. 
First of all we observe that for every $\varepsilon\in[0, 1)$ we have
 \begin{align}
   \label{eq:55}
   \begin{split}     
\sum_{k = 0}^{2^{l}-1}
\big|m\big((2^n+{2^{n-l}(k+1)})\xi\big)-m\big((2^n+{2^{n-l}k})\xi\big)\big|^{2-\varepsilon}\\
   \end{split}
 \end{align}
 \begin{align*}
  \lesssim \sum_{k = 0}^{2^{l}-1}\bigg(
 \int_{2^n+{2^{n-l}k}}^{2^n+{2^{n-l}(k+1)}}|\langle \xi, \nabla
 m(s\xi)\rangle|{\rm d}s \bigg)^{2-\varepsilon}
 \lesssim \sum_{k =
   0}^{2^{l}-1}\bigg(\log\bigg(\frac{2^n+2^{n-l}(k+1)}{2^n+2^{n-l}k}\bigg)\bigg)^{2-\varepsilon}
 \end{align*}
\begin{align*}
  \lesssim
\sum_{k =
	0}^{2^{l}-1}\bigg(\log\bigg(1+\frac{1}{2^l+k}\bigg)\bigg)^{2-\varepsilon}
\lesssim \frac{1}{2^{(1-\varepsilon)l}},
\end{align*}
where in the second inequality we have used \eqref{eq:69}.

Secondly, for every $0\le k\le 2^{l}$, we have
$(2^n+{2^{n-l}k})\simeq 2^n$, which guarantees
\begin{multline}
\label{eq:57}
  \big|m\big((2^n+{2^{n-l}(k+1)})\xi\big)-m\big((2^n+{2^{n-l}k})\xi\big)\big|^{\varepsilon}
|\mathcal
F S_{j+n}(\xi)|^2\\
\lesssim
\min\big\{|L2^n\xi|,|L2^n\xi|^{-1}\big\}^{\varepsilon}\big|e^{-2\pi
  L2^{j+n}|\xi|}-e^{-2\pi L2^{j+n-1}|\xi|}\big|^2\\
\lesssim 2^{-\varepsilon|j|/2}\min\big\{|L2^n\xi|,|L2^n\xi|^{-1}\big\}^{\varepsilon/2}.
\end{multline}
Combining \eqref{eq:55} with \eqref{eq:57} with $\varepsilon\in(0, 1)$ we immediately obtain 
\begin{multline}
  \label{eq:58}
\sum_{k = 0}^{2^{l}-1}
\big|m\big((2^n+{2^{n-l}(k+1)})\xi\big)-m\big((2^n+{2^{n-l}k})\xi\big)\big|^2
|\mathcal
F S_{j+n}(\xi)|^2\\
\lesssim 2^{-\varepsilon|j|/2}2^{-(1-\varepsilon)l}\min\big\{|L2^n\xi|,|L2^n\xi|^{-1}\big\}^{\varepsilon/2}.
\end{multline}
By the Plancherel theorem and \eqref{eq:58} we get
\begin{multline*}
  %\label{eq:85}
\bigg\|\Big(\sum_{n\in\ZZ} \sum_{k = 0}^{2^{l}-1} \big|M_{2^n+{2^{n-l}(k+1)}}S_{j+n}f
  - M_{2^n+{2^{n-l}k}}S_{j+n}f\big|^2
  \Big)^{1/2}\bigg\|_{L^2}^2\\
  \lesssim2^{-\varepsilon|j|/2}2^{-(1-\varepsilon)l}
  \int_{\RR^d}\sum_{n\in\ZZ}\min\big\{|L2^n\xi|,|L2^n\xi|^{-1}\big\}^{\varepsilon}|\mathcal F
  f(\xi)|^2{\rm d}\xi\lesssim 2^{-\varepsilon|j|/2}2^{-(1-\varepsilon)l}\|f\|_{L^2}^2.
\end{multline*}
This completes the proof of \eqref{eq:56}.\end{proof}
           
We now pass to the case when $p>2$.            
\begin{proof}[Proof of inequality \eqref{eq:91} for $p\in (2,4)$ in the settings of Theorem \ref{thm:6}]

         Note that by formulas
           \eqref{eq:20} and \eqref{eq:80} we obtain
           \begin{multline}
           \label{eq:19}
           \bigg\|\Big(\sum_{n\in\ZZ}V_2\big(M_{t}f: t\in[2^n,
           2^{n+1})\big)^2\Big)^{1/2}\bigg\|_{L^p}\\
           \lesssim \sum_{j\in\ZZ}\sum_{l\ge0}
           \bigg\|\Big(\sum_{n\in\ZZ}
           \sum_{k = 0}^{2^{l}-1}
           \big|M_{2^n+{2^{n-l}(k+1)}}S_{j+n}f - M_{2^n+{2^{n-l}k}}S_{j+n}f\big|^2
           \Big)^{1/2}\bigg\|_{L^p}\lesssim \|f\|_{L^p}.
           \end{multline}
           To establish the last inequality in \eqref{eq:19} we
           have to show that for every $p\in[2, 4)$, there are
           $\delta_p, \varepsilon_p>0$ such that 
           \begin{align}
           \label{eq:93}
           \bigg\|\Big(\sum_{n\in\ZZ}
           \sum_{k = 0}^{2^{l}-1}
           \big|M_{2^n+{2^{n-l}(k+1)}}S_{j+n}f - M_{2^n+{2^{n-l}k}}S_{j+n}f\big|^2
           \Big)^{1/2}\bigg\|_{L^p}\lesssim 2^{-\delta_pl}2^{-\varepsilon_p|j|}\|f\|_{L^p}
           \end{align}
           holds for all $f\in L^p(\RR^d)$ uniformly in $j\in\ZZ$
           and $l\ge0$.
           
            To this end, we show that, for $p\in [2,\infty)$, we have
            \begin{align}
\label{eq:59}
\bigg\|\Big(\sum_{n\in\ZZ} \sum_{k = 0}^{2^{l}-1} \big|M_{2^n+{2^{n-l}(k+1)}}S_{j+n}f
  - M_{2^n+{2^{n-l}k}}S_{j+n}f\big|^2
  \Big)^{1/2}\bigg\|_{L^p}
  \lesssim 2^{l/2}\|f\|_{L^p}.  
            \end{align}
 Then interpolation of \eqref{eq:56} with \eqref{eq:59}  does the job and we
 obtain \eqref{eq:93} for all $p\in[2, 4).$           
            
           Thus we focus on proving \eqref{eq:59}. Since $p\ge 2$ we estimate
            \begin{multline*}
\bigg\|\Big(\sum_{n\in\ZZ} \sum_{k = 0}^{2^{l}-1} \big|M_{2^n+{2^{n-l}(k+1)}}S_{j+n}f
  - M_{2^n+{2^{n-l}k}}S_{j+n}f\big|^2
  \Big)^{1/2}\bigg\|_{L^p}^2\\
  \le 2^l\max_{0\le k\le 2^l-1}\bigg\|\sum_{n\in\ZZ}  \big|M_{2^n+{2^{n-l}(k+1)}}S_{j+n}f
  - M_{2^n+{2^{n-l}k}}S_{j+n}f\big|^2
  \bigg\|_{L^{p/2}}\\
  \lesssim 2^l\max_{0\le k\le 2^l}\bigg\|\Big(\sum_{n\in\ZZ}  \big|M_{2^n+{2^{n-l}k}}S_{j+n}f\big|^2
  \Big)^{1/2}\bigg\|_{L^{p}}^2\lesssim 2^l\|f\|_{L^{p}}^2,
            \end{multline*}
            where the last inequality follows from \eqref{eq:81} and 
            \begin{align}
              \label{eq:31}
              \sup_{l\ge0}\max_{0\le k\le 2^l}\bigg\|\Big(\sum_{n\in\ZZ}  \big|M_{2^n+{2^{n-l}k}}g_n\big|^2
  \Big)^{1/2}\bigg\|_{L^{p}}\lesssim\bigg\|\Big(\sum_{n\in\ZZ}  |g_n|^2
  \Big)^{1/2}\bigg\|_{L^{p}}
            \end{align}
            which holds for all $p\in(1, \infty)$ and the implicit
            constant independent of the dimension. To prove
            \eqref{eq:31} we follow the argument used to justify
            \eqref{eq:29}.  This is feasible, since for every
            $p\in(1, \infty]$ and for
            every $f\in L^p(\RR^d)$ we have the following lacunary  estimate
            \[
\sup_{l\ge0}\max_{0\le k\le 2^l}\big\|\sup_{n\in\ZZ}|M_{2^n+{2^{n-l}k}}f|\big\|_{L^p}\le C_{p,\infty}\|f\|_{L^p}
\]
with the same constant  $C_{p, \infty}$ as in \eqref{eq:82}.
The last inequality can be established by appealing to  \eqref{eq:82}  with a new convex body $(1+2^{-l}k)G$, since
\[
M_{2^n+{2^{n-l}k}}f=M_{2^n+{2^{n-l}k}}^Gf=M_{2^n}^{(1+2^{-l}k)G}f.
\]
Hence, \eqref{eq:59} follows and the proof of \eqref{eq:91} for $p\in(2,4)$ is completed.
\end{proof}

We now pass to the case when $p<2$.

\begin{proof}[Proof of inequality \eqref{eq:91} for $p\in (3/2,2)$ in the settings of Theorem \ref{thm:6}]
Here we apply the almost orthogonality principle from Proposition
\ref{prop:2} with $p_0=3/2$, $Z=(0, \infty)$ and with $X=\RR^d$ endowed with the $\sigma$-algebra
$\mathcal B$ of all Lebesgue measurable sets on $\RR^d$ and the
Lebesgue measure $\mu=|\cdot|$ on $\RR^d$. Moreover, $T_t=M_t,$
$S_n=P_{L2^n}-P_{L2^{n-1}}$ is the Poisson projection and
$a_j=2^{-\varepsilon|j|/4}$ for some
    $\varepsilon\in(0, 1)$. Note that by \eqref{eq:20} and
\eqref{eq:80} the inequality from \eqref{eq:56} implies
\begin{equation}
\label{eq:56'}
\bigg\|\Big(\sum_{n\in\ZZ}V_2\big(M_{t}S_{j+n}: t\in[2^n,
2^{n+1})\big)^2\Big)^{1/2}\bigg\|_{L^2}
\lesssim  2^{-\varepsilon|j|/4}\|f\|_{L^2}.
\end{equation}
The above inequality together with \eqref{eq:81} and \eqref{eq:94}
show that in order to apply Proposition \ref{prop:2} it remains to
verify the estimate
\begin{align}
\label{eq:38'}
\sup_{n\in\ZZ}\big\|V_p\big(  M_t f: t\in[2^n, 2^{n+1})\big)\big\|_{L^p}\lesssim\|f\|_{L^p},
\end{align}
for $p\in (3/2,2]$. First note that by rescaling it suffices to prove that
\[
\big\|V_p\big(  M_t f: t\in[1, 2)\big)\big\|_{L^p}\lesssim\|f\|_{L^p}.
\]
Lemma \ref{lem:7}  gives for  any $\alpha\in (1/p, 1)$ that
\begin{equation*}
  %\label{eq:41}
  \begin{split}
  \big\|V_p\big(M_tf:t\in[1, 2)\big)\big\|_{L^p}&\le    \big\|V_p\big((\eta(t)M_tf):t\in\RR\big)\big\|_{L^p}\\
  &\lesssim\|f\|_{L^p}+\sup_{t>0}\big\|T_{((t\xi)\cdot\nabla)^{\alpha}m(t\xi)}f\big\|_{L^p},
  \end{split}
\end{equation*}
where $\eta$ is as in \eqref{eq:3}. Due to  \eqref{eq:54} we obtain
\begin{align*}
\sup_{t>0}\big\|T_{((t\xi)\cdot\nabla)^{\alpha}m(t\xi)}f\big\|_{L^p}\le  \sup_{t>0}\big\|T_{(\xi\cdot\nabla)^{\alpha}m}\big\|_{L^p\to L^p}\|f\|_{L^p}\lesssim \|f\|_{L^p}
\end{align*}
for $\alpha=2-2/p$ which combined with $\alpha>1/p$ gives $p>3/2$ and proves \eqref{eq:38'}.
Hence we are allowed to apply Proposition \ref{prop:2} and the inequality in \eqref{eq:91} is proven. 
\end{proof}
\subsection{Proof of Theorem \ref{thm:7}}
 Now $G$ is a ball induced by a small $\ell^q$ norm given in \eqref{eq:66}. Note that in this case we have
\begin{align}
  \label{eq:97}
  \big\|\sup_{t>0}|M_{t}f|\big\|_{L^p}\le C_{p, q, \infty}\|f\|_{L^p}
\end{align}
for all $f\in L^p(\RR^d)$ with the constant $C_{p, q, \infty}>0$
independent of the dimension. When $q\in [1,\infty)$ the bound \eqref{eq:97} is due to M\"uller \cite{Mul1}, while for $q=\infty$ it was proved by the first author \cite{B3}. 

Our aim will be
to show \eqref{eq:91} for all $p\in(1, \infty)$. This is enough by Theorem \ref{thm:3}. As before for simplicity
of notation we will write $M_t=M_t^G$ and
$m=m^G$.

The representation \eqref{eq:99} and the bound \eqref{eq:98} give a
useful estimate for the $L^p$ norm of the difference of $M_{t+h}f$ and
$M_t f$ for the convex bodies considered in Theorem \ref{thm:7}.
\begin{lemma}
	\label{lem:3}
	Fix $p\in(1, \infty)$ and $\alpha\in(1/2,
        1).$ Then, there exists a constant $C_{p,\alpha}>0$ such
        that for every $t,h>0,$ and for every $f\in L^p(\RR^d)$ we
        have
	\begin{align}
	\label{eq:61}
	\|M_{t+h}f - M_{t}f\|_{L^p}\lesssim C_{p,\alpha}\bigg(\frac{h}{t}\bigg)^{\alpha}
	\|f\|_{L^p}.
	\end{align}
	The same estimate remains true when the operator $M_t$ is replaced
	with its adjoint $M_t^*$.
\end{lemma}
\begin{proof}
It suffices to consider the case $h<t.$ After rescaling we may assume
that $t=1$ and $0<h<1.$ Now, by \eqref{eq:99} we have,
	$$ M_tf(x)=\int_0^{\infty}A(t,u)\mathcal P_u^{\alpha}f(x){\rm d}u,$$
	where we have set
        \begin{align*}
A(t,u)=
          \begin{cases}
          \frac1{\Gamma(\alpha)}\frac tu(1-\frac tu)^{\alpha-1} \frac1u,& \text{ if } u\ge t,\\
          \ \ \ \ \ \ \ \ \ \ 0, & \text{ if } u\le t.
          \end{cases}
        \end{align*}
                          Denote
                          $$X(u,h)=|A(1+h,u)-A(1,u)|.$$
Then,
	\begin{equation*}
	|M_{1+h}f(x)-M_1f(x)|\lesssim \int_1^{\infty}X(u, h)|\mathcal P_u^{\alpha}f(x)|{\rm d}u.
	\end{equation*}
	In view of Minkowski's integral inequality and \eqref{eq:98} we obtain
        \begin{align}
          \label{eq:110}
        \|M_{1+h}f-M_1f\|_{L^p}\leq  \sup_{u>1} \|\mathcal P_u^{\alpha}f\|_{L^p}\int_1^{\infty}X(u, h){\rm d}u\le \|f\|_{L^p}\int_1^{\infty}X(u, h){\rm d}u.  
        \end{align}
   Then simple calculations show that
        \begin{align*}
          X(u,h)\lesssim
          \begin{cases}
          |u-1|^{\alpha-1}, & \text{ if } 1\le u\le 1+h,\\
                    |u-h-1|^{\alpha-1}, & \text{ if } 1+h\le u\le 1+2h,\\
         |u-1|^{\alpha-2}h, & \text{ if } 1+2h\le u.
          \end{cases}
        \end{align*}
	Clearly
	$$\int_1^{\infty}X(u, h){\rm d}u\lesssim\int_1^{1+h}|u-1|^{\alpha-1}{\rm d}u+\int_{1+h}^{1+2h}|u-h-1|^{\alpha-1}{\rm d}u+
        h\int_{1+2h}^{\infty}|u-1|^{\alpha-2}{\rm d}u\lesssim h^{\alpha}.$$
	Hence, using \eqref{eq:110} we finish the proof of the lemma.
\end{proof}

The proof of \eqref{eq:91} is divided according to whether $p\in(2,\infty)$ or $p\in (1,2].$ We consider first the case when $p>2$.

\begin{proof}[Proof of \eqref{eq:91} for $p\in (2,\infty)$ in the settings of Theorem \ref{thm:7} ]
Here we would like to show that for every $\alpha\in(1/2,1)$ and
$p\in(2, \infty)$ there is
a constant $C_{\alpha, p}>0$ such that
\begin{align}
\label{eq:39}
\bigg\|\Big(\sum_{n\in\ZZ} \sum_{k = 0}^{2^{l}-1} \big|M_{2^n+{2^{n-l}(k+1)}}S_{j+n}f
- M_{2^n+{2^{n-l}k}}S_{j+n}f\big|^2
\Big)^{1/2}\bigg\|_{L^p}
\le C_{\alpha, p} 2^{\frac{l}2-\frac{p'}{2}\alpha l}\|f\|_{L^p}.    
\end{align}
Then, taking $\alpha$ sufficiently close to $1$ and interpolating with
\eqref{eq:56} we complete the proof. Observe
that
\begin{multline}
\label{eq:132}
\bigg\|\Big(\sum_{n\in\ZZ} \sum_{k = 0}^{2^{l}-1} \big|M_{2^n+{2^{n-l}(k+1)}}S_{j+n}f
- M_{2^n+{2^{n-l}k}}S_{j+n}f\big|^2
\Big)^{1/2}\bigg\|_{L^p}\\
\le
2^{l/2}\max_{0\le k\le 2^{l}-1}\bigg\|\Big(\sum_{n\in\ZZ}  \big|M_{2^n+{2^{n-l}(k+1)}}S_{j+n}f
- M_{2^n+{2^{n-l}k}}S_{j+n}f\big|^2\Big)^{1/2}
\bigg\|_{L^{p}}
\end{multline}
since $p>2$. Take any sequence $(g_n: n\in\ZZ)\in L^{p'}(\ell^2)$ such
that
\[
\bigg\|\Big(\sum_{n\in\ZZ} |g_n|^2
\Big)^{1/2}\bigg\|_{L^{p'}}\le 1.
\]
Then by the duality it will suffice to prove, for every $p'\in(1,2)$, that
\begin{align}
\label{eq:133}
\max_{0\le k\le 2^{l}-1}\bigg\|\Big(\sum_{n\in\ZZ}  \big|M_{2^n+{2^{n-l}(k+1)}}^*g_n
- M_{2^n+{2^{n-l}k}}^*g_n\big|^2\Big)^{1/2}
\bigg\|_{L^{p'}}\lesssim 2^{-\frac{p'}{2} \alpha l}     \bigg\|\Big(\sum_{n\in\ZZ} |g_n|^2
\Big)^{1/2}\bigg\|_{L^{p'}}.
\end{align}

Lemma \ref{lem:3} will be critical in the proof of  \eqref{eq:133}.
We shall prove that for every $q\in(1, 2)$ we have
\begin{align}
\label{eq:138}
\bigg\|\Big(\sum_{n\in\ZZ}  \big|M_{2^n+{2^{n-l}(k+1)}}^*g_n
- M_{2^n+{2^{n-l}k}}^*g_n\big|^{2}\Big)^{1/{2}}
\bigg\|_{L^{q}}\lesssim 2^{-\frac{q}{2}\alpha l}     \bigg\|\Big(\sum_{n\in\ZZ} |g_n|^{2}
\Big)^{1/{2}}\bigg\|_{L^{q}}.
\end{align}
Taking $q=p'$ in \eqref{eq:138} we obtain \eqref{eq:133}. 
To prove the estimate \eqref{eq:138} we will use a vector-valued
interpolation between $L^q(\ell^q)$ and $L^q(\ell^{\infty})$ for
$q\in(1, 2)$. Take $\theta\in (0, 1)$ satisfying 
\[
\frac{1}{2}=\frac{\theta}{q}+\frac{1-\theta}{\infty};
\]
then $\theta=q/2$. For $L^q(\ell^q)$ we have
\begin{align}
\label{eq:136}
\bigg\|\Big(\sum_{n\in\ZZ}  \big|M_{2^n+{2^{n-l}(k+1)}}^*g_n
- M_{2^n+{2^{n-l}k}}^*g_n\big|^{q}\Big)^{1/{q}}
\bigg\|_{L^{q}}\lesssim 2^{-\alpha l}     \bigg\|\Big(\sum_{n\in\ZZ} |g_n|^{q}
\Big)^{1/{q}}\bigg\|_{L^{q}}    
\end{align}
since, by \eqref{eq:61} it holds
\[
\big\|M_{2^n+{2^{n-l}(k+1)}}^*f
- M_{2^n+{2^{n-l}k}}^*f
\big\|_{L^{q}}\lesssim  2^{-\alpha l}\|f\|_{L^{q}}.
\]
The $L^q(\ell^{\infty})$ endpoint is estimated
using \eqref{eq:97} as
\begin{align}
\label{eq:137}
\begin{split}
\big\|\sup_{n\in\ZZ} \big| M_{2^n+{2^{n-l}(k+1)}}^*g_n
&  - M_{2^n+{2^{n-l}k}}^*g_n\big|
\big\|_{L^{q}}\\
&\lesssim \big\|\sup_{n\in\ZZ}  M_{2^n+{2^{n-l}(k+1)}}^*g\big\|_{L^{q}}
+\big\| \sup_{n\in\ZZ}M_{2^n+{2^{n-l}k}}^*g
\big\|_{L^{q}}      \lesssim \|g\|_{L^q}             
\end{split}
\end{align}
where $g(x)=\sup_{n\in\ZZ}|g_n(x)|$. Now, invoking 
interpolation we obtain \eqref{eq:138}.
\end{proof}

Now we pass to the case when $p<2$.

\begin{proof}[Proof of \eqref{eq:91} for {$p\in(1, 2]$} in the settings of Theorem \ref{thm:7} ]	
Now our aim will be to prove that for every $p\in(1, 2)$ there is a constant $C_p>0$
  independent of $d$ such that
  \begin{align}
    \label{eq:26}
     \bigg\|\Big(\sum_{n\in\ZZ}V_2\big(M_{t}f: t\in[2^n,
   2^{n+1})\big)^2\Big)^{1/2}\bigg\|_{L^p}\le C_p \|f\|_{L^p}
  \end{align}
  for all $f\in L^p(\RR^d)$. For this purpose we will again apply
 our almost orthogonality principle for $r$-variations  
  Proposition \ref{prop:2} with $p_0=1$,
  $Z=(0, \infty),$ and with $X=\RR^d$ endowed with the $\sigma$-algebra
  $\mathcal B$ of all Lebesgue measurable sets on $\RR^d$ and the
  Lebesgue measure $\mu=|\cdot|$ on $\RR^d$. Moreover, we take $T_t=M_t,$
  $S_n=P_{L2^n}-P_{L2^{n-1}},$ and
  $a_j=2^{-\varepsilon|j|/4}$ for some $\varepsilon\in(0, 1)$.
  Condition \eqref{eq:92} was already justified in \eqref{eq:56'}. It
  only remains to verify condition \eqref{eq:33}, which in our case
  says that for every $p\in(1, 2]$  there is a
  constant $C_{p}>0$ independent of $d$ such that
   \begin{align}
     \label{eq:37}
         \sup_{n\in\ZZ}\big\|V_p\big(  M_t f: t\in[2^n,
     2^{n+1})\big)\big\|_{L^p}\le C_{p}\|f\|_{L^p}
   \end{align}
   holds for all $f\in L^p(\RR^d)$. Once \eqref{eq:37} is proven then
   Proposition \ref{prop:2} applies and completes the proof of Theorem
   \ref{thm:7}. By rescaling it suffices to prove that
\[
\big\|V_p\big(  M_t f: t\in[1, 2)\big)\big\|_{L^p}\lesssim\|f\|_{L^p}.
\]
By Lemma \ref{lem:7}  we obtain for  any $\alpha\in (1/p, 1)$ that
\begin{equation*}
  %\label{eq:41}
  \begin{split}
  \big\|V_p\big(M_tf:t\in[1, 2)\big)\big\|_{L^p}&\le    \big\|V_p\big((\eta(t)M_tf):t\in\RR\big)\big\|_{L^p}\\
  &\lesssim\|f\|_{L^p}+\sup_{t>0}\big\|T_{((t\xi)\cdot\nabla)^{\alpha}m(t\xi)}f\big\|_{L^p},
  \end{split}
\end{equation*}
where $\eta$ is as in \eqref{eq:3}. Due to   \eqref{eq:98'} for $\alpha\in(1/2, 1)$ we obtain
\begin{align*}
  \sup_{t>0}\big\|T_{((t\xi)\cdot\nabla)^{\alpha}m(t\xi)}f\big\|_{L^p}\le  \sup_{t>0}\big\|T_{(\xi\cdot\nabla)^{\alpha}m}\big\|_{L^p\to L^p}\|f\|_{L^p}\lesssim \|f\|_{L^p}.
\end{align*}
If $p\in(1, 2)$ is close to $1$ we can always take an $\alpha$ such that
$\alpha>1/p$ and the proof of \eqref{eq:37} is completed.
\end{proof}

\section{Transference principle}
\label{sec:tp}
In this section we prove the transference principle, which will allow
us to deduce estimates for $r$-variations on $L^p(X, \mu)$ for the
operator $A_t^G$ from the corresponding bounds for $M_t^G$ on
$L^p(\RR^d)$. Specifically, in view of Proposition \ref{prop:3}, the
inequalities from \eqref{eq:43}, \eqref{eq:60} and \eqref{eq:63} will
follow respectively from \eqref{eq:65}, \eqref{eq:64} and \eqref{eq:67}.

\begin{proposition}
\label{prop:3}
Suppose that for some $p\in(1, \infty)$ and $r\in(2, \infty]$ there is a constant $C_{p, r}>0$ such that for every $d\in \NN$ and for every symmetric convex body $G\subset\RR^d$  the following estimate
\begin{align}
  \label{eq:101}
  \big\|V_r\big(  M_t^G h: t\in Z)\big)\big\|_{L^p(\RR^d)}\le C_{p, r}\|h\|_{L^p(\RR^d)}
\end{align}
holds for all $h\in L^p(\RR^d)$, where $Z\subseteq (0, \infty)$.
Let $A_t^G$ be the ergodic counterpart of $M_t^G$ defined in
\eqref{eq:22} for a given $\sigma$-finite measure space
$(X, \mathcal B, \mu)$ with families of commuting and
measure-preserving transformations $(T_1^t: t\in\RR), \ldots, (T_d^t: t\in\RR)$, which map $X$ to itself.

Then for every  $f\in L^p(X, \mu)$ the inequality
\begin{align}
  \label{eq:105}
  \big\|V_r\big(  A_t^G f: t\in Z)\big)\big\|_{L^p(X, \mu)}\le C_{p,r}\,\|f\|_{L^p(X, \mu)}
\end{align}
holds  with the parameters $p$, $r,$ and the constant $C_{p,r}$ as in \eqref{eq:101}.

\end{proposition}

\begin{proof}
We fix $f\in L^p(X, \mu)$, $\varepsilon>0,$ and $R>0,$ and define for every $x\in X$ the auxiliary function
\[
\phi_x(y)=
\begin{cases}
f\big(T_1^{y_1}\circ\ldots\circ T_d^{y_d}x\big), & \text{ if } y\in G_{R(1+\varepsilon/d)},\\
0, & \text{ otherwise}.
\end{cases}
\]
Then, for every $z\in G_R$ and $t<R\varepsilon/d,$ we have
\begin{align}
  \label{eq:106}
  \begin{split}
  A_t^Gf\big(T_1^{z_1}\circ\ldots\circ T_d^{z_d}x\big)&=
                                                                       \frac{1}{|G_t|}\int_{G_t}f\big(T_1^{z_1-y_1}\circ\ldots\circ T_d^{z_d-y_d    }x\big){\rm d}y_1\ldots{\rm d}y_d\\
  &=\frac{1}{|G_t|}\int_{G_t}\phi_x(z-y){\rm d}y=M_t^{G}\phi_x(z).
  \end{split}
\end{align}
Indeed, since $G$ is a symmetric convex body we have that  $z-y\in G_{R(1+\varepsilon/d)}$, whenever $z\in G_R$ and $y\in G_{t}.$ Hence, by \eqref{eq:106} and \eqref{eq:101} we get
\begin{align}
  \label{eq:108}
  \begin{split}
  \int_{G_R}\big|V_r\big(A_t^Gf\big(T_1^{z_1}\circ&\ldots\circ T_d^{z_d}x\big): t\in Z\cap(0, R\varepsilon/d)\big)\big|^p{\rm d}z_1\ldots{\rm d}z_d\\
  &  \le   \int_{G_R}\big|V_r\big(M_t^G\phi_x(z): t\in
  Z\cap(0, R\varepsilon/d)\big)\big|^p{\rm d}z_1\ldots{\rm d}z_d\\
  &\le \big\|V_r\big(M_t^G\phi_x: t\in Z\big)\big\|_{L^p(\RR^d)}^p\\
  &\le C_{p, r}^p\|\phi_x\|_{L^p(\RR^d)}^p.
  \end{split}
\end{align}
Averaging \eqref{eq:108} over $x\in X$ we obtain
\begin{align}
  \label{eq:109}
  \begin{split}
  \int_{G_R}\big\|V_r\big(A_t^Gf\big(T_1^{z_1}\circ&\ldots\circ T_d^{z_d}x\big): t\in Z\cap(0, R\varepsilon/d)\big)\big\|_{L^p(X, \mu)}^p{\rm d}z_1\ldots{\rm d}z_d\\
  &\le  C_{p, r}^p\int_{G_{R(1+\varepsilon/d)}}\big\|f\big(T_1^{z_1}\circ\ldots\circ T_d^{z_d}x\big)\big\|_{L^p(X, \mu)}^p{\rm d}z_1\ldots{\rm d}z_d,
  \end{split}
\end{align}
by definition of $\phi_x$, which is supported on $G_{R(1+\varepsilon/d)}$. Inequality \eqref{eq:109} guarantees that
\begin{align*}
  |G_R|\cdot \big\|V_r\big(A_t^Gf: t\in Z\cap(0, R\varepsilon/d)\big)\big\|_{L^p(X, \mu)}^p
  \le C_{p, r}^p\cdot|G_{R(1+\varepsilon/d)}|\cdot \|f\|_{L^p(X, \mu)}^p,
\end{align*}
since all $T_1^{z_1},\ldots,T_d^{z_d}$ preserve the measure $\mu$ on $X$. Dividing both sides by $|G_{R}|$ we obtain
that
\begin{equation*}
%\label{eq:109'}
\begin{split}
  \big\|V_r\big(A_t^Gf: t\in Z\cap(0, R\varepsilon/d)\big)\big\|_{L^p(X, \mu)}^p
  \le  C_{p, r}^p(1+\varepsilon/d)^d\|f\|_{L^p(X, \mu)}^p\le C_{p, r}^p\,e^{\varepsilon}\|f\|_{L^p(X, \mu)}^p.
  \end{split}
\end{equation*}
This is the place where the parameter $\varepsilon>0$ is helpful.
Namely, taking $R\to\infty$ and invoking the monotone convergence theorem we conclude that
$$\big\|V_r\big(  A_t^G f: t\in Z)\big)\big\|_{L^p(X, \mu)}\le C_{p,r}\, e^{\varepsilon/p}\|f\|_{L^p(X, \mu)},$$
with arbitrary $\varepsilon>0.$
Therefore, letting $\varepsilon\to 0^+$ we obtain \eqref{eq:105} and complete the proof of the proposition.
\end{proof}

\appendix
\section{Dimension-free bounds for the ball averages}
We give a different proof of the dimension-free estimate for the
averages over the Euclidean balls in the full range of
$p\in(1, \infty)$ and $r\in(2, \infty).$ The proof is in spirit of the
original proof for the maximal function from \cite{SteinMax}. Here by
$M_t$ we always mean $M_t^G$ with $G$ being the Eucledean ball in
$\RR^d.$ The main result of this section reads as follows.
\begin{theorem}
\label{thm:ball} For $p\in(1, \infty)$ and $r\in(2, \infty)$ there is $C_{p, r}>0$ independent of $d$ such that
the inequality
	\[
	\big\lVert
	V_r\big(  M_t f: t>0\big)
	\big\rVert_{L^p}\le
	C_{r,p}\|f\|_{L^p},
	\]
        holds for all $f\in L^p(\RR^d)$.
\end{theorem} 
From Section \ref{sec:3} we know that the estimate holds for long
variations. Thus, it is enough to provide the bound for short variations
$$
\bigg\lVert\Big(\sum_{n\in\ZZ}
V_2\big(  M_t f: t\in[2^n, 2^{n+1})\big)^2\Big)^{1/2}
\bigg\rVert_{L^p}\le C_p\|f\|_{L^p}.
$$
The estimates of short variations in Theorem \ref{thm:ball} will be based
on \eqref{eq:42} from the next lemma.

\begin{lemma}
	\label{lem:2}
	Let $u < v$ be real numbers and $\mathfrak a:[u, v]\rightarrow [0, \infty)$ be a
	differentiable function. For  every $r\in[2, \infty)$ we have
	\begin{align}
	\label{eq:15}
	V_r\big(\mathfrak a_t: t\in[u, v)\big)\lesssim
	\bigg(\int_u^v|\mathfrak a_t|^2\frac{{\rm
			d}t}{t}\bigg)^{1/4} \cdot
	\bigg(\int_u^v\Big|t\frac{{\rm
			d}}{{\rm d}t}\mathfrak
	a_t\Big|^2\frac{{\rm d}t}{t}\bigg)^{1/4}.
	\end{align}
	If additionally $u=2^l$ and $v=2^{l+1}$ for some $l\in\ZZ$ then
	\begin{align}
	\label{eq:17}
	V_r\big(\mathfrak a_t: t\in[2^l, 2^{l+1})\big)\lesssim
	\bigg(\int_{2^l}^{2^{l+1}}\Big|t\frac{{\rm
			d}}{{\rm d}t}\mathfrak
	a_t\Big|^2\frac{{\rm
			d}t}{t}\bigg)^{1/2}.
	\end{align}
	Consequently,
	\begin{equation}
	\label{eq:42}
	\bigg( \sum_{n\in \ZZ} V_2\big(\mathfrak a_t: t\in[2^n, 2^{n+1})\big)^2\bigg)^{1/2} \lesssim
	\bigg(\int_{0}^{\infty}\Big|t\frac{{\rm
			d}}{{\rm d}t}\mathfrak
	a_t\Big|^2\frac{{\rm d}t}{t}\bigg)^{1/2}
	\end{equation}
\end{lemma}
\begin{proof}
	For the proof of \eqref{eq:15}, let $u\le t_0<t_1<\ldots<t_M<v$ be
	an arbitrary increasing sequence. By, the mean value
	theorem and the Cauchy--Schwarz inequality we obtain
	
	\begin{align*}
	|\mathfrak a_{t_{j+1}}-\mathfrak a_{t_{j}}|^2&\le
	|\mathfrak a_{t_{j+1}}^2-\mathfrak
	a_{t_{j}}^2|=\Big|\int_{t_{j}}^{t_{j+1}}2\mathfrak
	a_tt\frac{{\rm
			d}}{{\rm d}t}\mathfrak
	a_t\frac{{\rm
			d}t}{t}\Big|\\
	&\le 2
	\bigg(\int_{t_{j}}^{t_{j+1}}|\mathfrak a_t|^2\frac{{\rm
			d}t}{t}\bigg)^{1/2}\cdot
	\bigg(\int_{t_{j}}^{t_{j+1}}\Big|t\frac{{\rm
			d}}{{\rm d}t}\mathfrak
	a_t\Big|^2\frac{{\rm d}t}{t}\bigg)^{1/2}.
	\end{align*}
	Therefore, by the Cauchy--Schwarz inequality, we see
	\begin{align*}
	\sum_{j=0}^{M-1}|\mathfrak a_{t_{j+1}}-\mathfrak a_{t_{j}}|^2\le &2\sum_{j=0}^{M-1}
	\bigg(\int_{t_{j}}^{t_{j+1}}|\mathfrak a_t|^2\frac{{\rm
			d}t}{t}\bigg)^{1/2}\cdot
	\bigg(\int_{t_{j}}^{t_{j+1}}\Big|t\frac{{\rm
			d}}{{\rm d}t}\mathfrak
	a_t\Big|^2\frac{{\rm
			d}t}{t}\bigg)^{1/2}\\
	&\le 2
	\bigg(\int_u^v|\mathfrak a_t|^2\frac{{\rm d}t}{t}\bigg)^{1/2} \cdot
	\bigg(\int_u^v\Big|t\frac{{\rm
			d}}{{\rm d}t}\mathfrak
	a_t\Big|^2\frac{{\rm d}t}{t}\bigg)^{1/2}
	\end{align*}
	as desired. To prove \eqref{eq:17} observe that by the mean value
	theorem and the Cauchy--Schwarz inequality we get 
	\begin{align*}
	|\mathfrak a_{t_{j+1}}-\mathfrak
	a_{t_{j}}|^2=\Big|\int_{t_{j}}^{t_{j+1}}t\frac{{\rm
			d}}{{\rm d}t}\mathfrak
	a_t\frac{{\rm
			d}t}{t}\Big|^2
	\le 
	(t_{j+1}-t_j)
	\int_{t_{j}}^{t_{j+1}}\Big|t\frac{{\rm
			d}}{{\rm d}t}\mathfrak
	a_t\Big|^2\frac{{\rm d}t}{t^2}\le  
	\int_{t_{j}}^{t_{j+1}}\Big|t\frac{{\rm
			d}}{{\rm d}t}\mathfrak
	a_t\Big|^2\frac{{\rm d}t}{t}
	\end{align*}
	for any sequence $2^l\le t_0<t_1<\ldots<t_M<2^{l+1}$, since
	$t_{j+1}-t_j\le 2^l$.
	Thus
	\begin{align*}
	\sum_{j=0}^{M-1}|\mathfrak a_{t_{j+1}}-\mathfrak
	a_{t_{j}}|^2\le
	\int_{2^l}^{2^{l+1}}\Big|t\frac{{\rm
			d}}{{\rm d}t}\mathfrak
	a_t\Big|^2\frac{{\rm d}t}{t}
	\end{align*}
	and the proof is completed. 
\end{proof}
Observe now that in view of \eqref{eq:42} the estimate for short variations reduces to 
\begin{equation}
\label{eq:44}
\bigg\|\bigg(\int_0^{\infty}\Big|t\frac{\rm d}{{\rm d}t}M_t f\Big|^2\frac{{\rm d}t}{t}\bigg)^{1/2}\bigg\|_{L^p}\leq C_p \|f\|_{L^p}.\end{equation}

Let 
$$S_tf(x)=\int_{\mathbb S^{d-1}}f(x-ty){\rm d}\sigma(y),$$
where ${\rm d}\sigma$ denotes the normalized spherical measure on the unit sphere $\mathbb S^{d-1}$ of $\RR^d.$
The estimate \eqref{eq:44} will be deduced from an analogous statement for the spherical averages, namely
\begin{equation}
\label{eq:45}
\bigg\|\bigg(\int_0^{\infty}\Big|t\frac{{\rm d}}{{\rm d}t}S_t f\Big|^2\,\frac{{\rm d}t}{t}\bigg)^{1/2}\bigg\|_{L^p}\leq C_p \|f\|_{L^p}.\end{equation} To see that \eqref{eq:45} does imply \eqref{eq:44} we note that
$$M_tf(x)=d\int_0^1 u^{d-1}S_{tu}f(x){\rm d}u.$$ Thus, setting $F(t,x)=S_tf(x)$ we have
$$\frac{{\rm d}}{{\rm d}t}M_tf(x)= d\int_0^1 u^{d-1}u(\partial_1F)(ut,x){\rm d}u,$$
and, consequently, by Minkowski's inequality
\begin{align*}
  \bigg\|\bigg(\int_0^{\infty}\Big|t\frac{{\rm d}}{{\rm d}t}M_t f\Big|^2\,\frac{{\rm d}t}{t}\bigg)^{1/2}\bigg\|_{L^p}&\leq d\int_0^{1}u^{d-1}\left\|\bigg(\int_0^{\infty}\big|ut(\partial_1F)(ut,x)\big|^2\,\frac{{\rm d}t}{t}\bigg)^{1/2}\right\|_{L^p}{\rm d}u\\
&= \bigg\|\bigg(\int_0^{\infty}\Big|t\frac{{\rm d}}{{\rm d}t}S_t f\Big|^2\,\frac{{\rm d}t}{t}\bigg)^{1/2}\bigg\|_{L^p}.
\end{align*}
From now on we focus on \eqref{eq:45}. More precisely, we shall prove the following proposition.

\begin{proposition}
	\label{pro:spher_sq}
	For each $p\in(1, \infty)$ there exists $d_0(p)\in\NN$
        such that for $d\ge d_0(p)$ we have
	$$ \bigg\|\bigg(\int_0^{\infty}\Big|t\frac{{\rm d}}{{\rm d}t}S_t f\Big|^2\,\frac{{\rm d}t}{t}\bigg)^{1/2}\bigg\|_{L^p}\le C_{p}\|f\|_{L^p}$$
	for every $f\in L^p(\RR^d)$, where $C_p>0$ is independent of $d$.
\end{proposition}
To see that Proposition \ref{pro:spher_sq} is enough we note that it
implies \eqref{eq:44}, hence, also Theorem \ref{thm:ball} for
dimensions $d\ge d_0(p).$ For the finite number of smaller dimensions
$d<d_0(p)$ we just use the known fact that Theorem \ref{thm:ball}
holds with some $C_{r,p,d}>0$.

We pass to the proof of the proposition.
\begin{proof}[Proof of {Proposition \ref{pro:spher_sq}}]
Set
\begin{align*}
  K^{\alpha}(x)=
  \begin{cases}
  \frac1{\Gamma(\alpha)}(1-|x|^2)^{\alpha-1}, & \text{ for } |x|<1,\\
  \ \ \ \ \ \ \ \ \ 0 & \text{ for } |x|\ge1,
  \end{cases}
\end{align*}
	and let $m^{\alpha}(\xi)=\calF(K^{\alpha})(\xi).$ Then $K^{\alpha},$ as a function of $\alpha\in \CC,$ is analytic. 
	
	For $t>0$ we define $K_t^{\alpha}(x)=t^{-d}K^{\alpha}(x/t)$,  thus $\calF(K_t^{\alpha})(\xi)=m(t\xi).$ It is well known that 
	$$m^{\alpha}(\xi)=\pi^{-\alpha+1} |\xi|^{-d/2-\alpha+1}J_{d/2+\alpha-1}(2\pi |\xi|),$$
	where $J_{\nu}$ is the Bessel function of order $\nu.$ Therefore
	\begin{equation}
	\label{eq:46}
	|m^{\alpha}(\xi)|+|\nabla m^{\alpha}(\xi)|\le C_{d,\Real(\alpha)}\min(1,|\xi|^{-d/2+1/2-\Real (\alpha)}).\end{equation}
	
	Set $S_t^{\alpha}f=f*K_{t}^{\alpha}$ so that $S_t^0=S_t.$
	Using \eqref{eq:46} and a straightforward computation we obtain
	$$\int_0^{\infty}\Big|t\frac{{\rm d}}{{\rm d}t}m^{\alpha}(t\xi)\Big|^2\,\frac{{\rm d}t}{t}\leq C_{d,\Real(\alpha)},$$
	for $\Real(\alpha)>(3-d)/2.$ 
	Consequently, for such $\alpha$ we have
	\begin{equation}
	\label{eq:47}
	\bigg\|\bigg(\int_0^{\infty}\Big|t\frac{{\rm d}}{{\rm d}t}S_t^{\alpha} f\Big|^2\,\frac{{\rm d}t}{t}\bigg)^{1/2}\bigg\|_{L^2}\le C_{d,\Real(\alpha)}\|f\|_{L^2}.
	\end{equation} 
	We claim that, if $\Real(\alpha)=3,$ then for all $1<p<\infty$ it holds 
	\begin{equation}
	\label{eq:48}
	\bigg\|\bigg(\int_0^{\infty}\Big|t\frac{{\rm d}}{{\rm d}t}S_t^{\alpha} f\Big|^2\,\frac{{\rm d}t}{t}\bigg)^{1/2}\bigg\|_{L^p}\le \frac{C_{d,p}}{|\Gamma(\alpha)|}\|f\|_{L^p}.
	\end{equation} 
	To prove \eqref{eq:48} we use the vector-valued
        Calder\'on--Zygmund theory. Namely, we consider the
        operator
        $$Tf=f*\calK,\qquad \text{where} \qquad \calK(x)=\bigg(t\frac{{\rm d}}{{\rm
            d}t}K_t^{\alpha}(x): t>0\bigg)$$ as a mapping with values
        in the Hilbert space
        $H=L^2\big((0,\infty),\,\frac{{\rm d}t}{t}\big).$ Then \eqref{eq:48}
        is equivalent to proving that $T$ maps $L^p$ into $L^p(H).$
        Note that \eqref{eq:47} implies that this is true for $p=2.$
        Thus, the vector-valued Calder\'on--Zygmund theory reduces our
        task to checking that
	\begin{align*}
	\|\calK(x)\|_{H}\leq \frac{C_{d}}{|\Gamma(\alpha)|}|x|^{-d},\qquad \|\nabla_{x}\calK(x)\|_{H}\leq \frac{C_{d}}{|\Gamma(\alpha)|}|x|^{-d-1}.
	\end{align*} 
	When $\Real(\alpha)=3$ both of these estimates (which actually are equalities) follow from the definition of $K_{t}^{\alpha}$. Hence, \eqref{eq:48} is proved.
	
	In the next step we fix $1<p<\infty.$ We will show the existence of $d_0=d_0(p)$ such that
	\begin{equation} 
	\label{eq:49}
	\bigg\|\bigg(\int_0^{\infty}\Big|t\frac{{\rm d}}{{\rm d}t}S_t f\Big|^2\,\frac{{\rm d}t}{t}\bigg)^{1/2}\bigg\|_{L^p(\RR^{d_0})}\le C_{p}\|f\|_{L^p(\RR^{d_0})}.\end{equation}
	Assume first that $1<p<2.$ Take $1<q<2,$ where $q$ is close to $1$ and write $1/p=(1-\theta)/2+\theta/q,$ where $\theta\in (0,1).$ We also let $\varepsilon$ be a small real number. The quantities $q,\theta,$ and $\varepsilon$ will be determined in a moment.  We use complex interpolation for the analytic family of operators
        $$\calS^{\alpha}f=\bigg(e^{\alpha^2}t\frac{{\rm d}}{{\rm d}t}S_t^{\alpha}f: t>0\bigg).$$
        Namely, by \eqref{eq:47} and \eqref{eq:48} we have \begin{equation}
	\label{eq:50}
	\begin{split}\|\calS^{\alpha}\|_{L^2(\RR^d,H)\to L^2(\RR^d)}&\le C_{d,\varepsilon} \qquad  \textrm{for }\Real(\alpha)=(3-d)/2+\varepsilon\\
	\|\calS^{\alpha}\|_{L^q(\RR^d,H)\to L^q(\RR^d)}&\le C_{d,q}\qquad \textrm{for }\Real(\alpha)=3.
	\end{split}\end{equation} 
	We want to take $\varepsilon$ positive and such that $0=(1-\theta)((3-d)/2+\varepsilon)+3\theta.$ Then
	\begin{equation}
	\label{eq:51}
	p=((1-\theta)/2+\theta/q)^{-1}>\frac{(d+3)}{(d-3)/q+3}.\end{equation}
	Based on the above considerations we claim that \eqref{eq:49} holds for any $d_0=d_0(p)>3/(p-1).$
	Indeed, then $p>\frac{d_0(p)+3}{d_0(p)},$ so that there is a $1<q<p$ small enough and such that \eqref{eq:51} holds for $q$ and $d_0(p)$. But then \eqref{eq:50} holds with a positive $\varepsilon,$ hence   
	by complex interpolation in \eqref{eq:50}, for $1<p<2$ we obtain the bound
	$$ \bigg\|\bigg(\int_0^{\infty}\Big|t\frac{{\rm d}}{{\rm
            d}t}S_t f\Big|^2\,\frac{{\rm
            d}t}{t}\bigg)^{1/2}\bigg\|_{L^p(\RR^{d_0})}\le
        C_{d,p}\|f\|_{L^p(\RR^{d_0})}.$$  If $p\ge 2,$ then a similar
        interpolation argument shows that \eqref{eq:49} holds for any
        $d_0(p)>3/(p'-1).$ Therefore, \eqref{eq:49} is proved for
        $d_0=d_0(p)= \lceil \max\{3/(p-1), 3/(p'-1)\}\rceil$. 
	
	In what follows we fix $d_0$ such that \eqref{eq:49} holds. Let $d\ge d_0$ and, for $$x=(x_1, x_2)\in \RR^{d}=\RR^{d_0}\times\RR^{d-d_0}$$ define
	$$S_t'f(x_1, x_2)=\int_{\mathbb S^{d_0-1}}f(x_1-ty_1, x_2)\,d\sigma(y_1).$$
	Since $S_t'$ acts only on $x_1$ from \eqref{eq:49} we obtain
	\begin{equation} 
	\label{eq:52}
	\bigg\|\bigg(\int_0^{\infty}\Big|t\frac{{\rm d}}{{\rm d}t}S_t' f\Big|^2\,\frac{{\rm d}t}{t}\bigg)^{1/2}\bigg\|_{L^p(\RR^{d})}\le C_{p}\|f\|_{L^p(\RR^{d})}.\end{equation}
	Note that the constant $C_p$ from \eqref{eq:49} is preserved. Let $\mathcal O(d)$ be the group of orthogonal transformations on $\RR^d$.
	For $\rho\in \mathcal O(d)$ we set
	$$S_t^{\rho}=\rho\circ S_t'\circ \rho^{-1},\qquad \textrm{where }\qquad (\rho\circ g)(x)=g(\rho x).$$ Then
	$\frac{{\rm d}}{{\rm d}t}(S_t^{\rho})=\rho\circ (\frac{{\rm
            d}}{{\rm d}t}S_t')\circ \rho^{-1},$ and \eqref{eq:52}
        gives, for $\rho\in \mathcal O(d),$ the bound
	\begin{equation} 
	\label{eq:53}
	\bigg\|\bigg(\int_0^{\infty}\Big|t\frac{{\rm d}}{{\rm d}t}S_t^{\rho} f\Big|^2\,\frac{{\rm d}t}{t}\bigg)^{1/2}\bigg\|_{L^p(\RR^{d})}\le C_{p}\|f\|_{L^p(\RR^{d})}.\end{equation}
	To complete the proof we note that
	$$S_tf(x)=\int_{\mathcal O(d)}S_t^{\rho}f(x){\rm d}\rho,$$
	where ${\rm d}\rho$ denotes the probabilistic Haar measure on $\mathcal O(d)$.
	Therefore
	$$t\,\frac{{\rm d}}{{\rm d}t}S_tf(x)=\int_{\mathcal O(d)}t\,\frac{{\rm d}}{{\rm d}t}\,S_t^{\rho}f(x){\rm d}\rho.$$ Hence, an application of Minkowski's integral inequality for the $L^p(H)$ norm in question together with \eqref{eq:53} leads to 
	$$ \bigg\|\bigg(\int_0^{\infty}\Big|t\frac{{\rm d}}{{\rm d}t}S_t f\Big|^2\,\frac{{\rm d}t}{t}\bigg)^{1/2}\bigg\|_{L^p}\le C_{p}\|f\|_{L^p}.$$
	This finishes the proof of the proposition.
\end{proof}

\end{document}